\documentclass[11pt,reqno,tbtags,draft]{amsart}
\usepackage{amssymb}
\usepackage{url}
\usepackage[square,numbers]{natbib}
\bibpunct[, ]{[}{]}{;}{n}{,}{,}

\title[Tensor norms and exchangeable distributions]
{Tensor norms on ordered normed spaces,
polarization constants,  
and exchangeable distributions}

\date{6 November, 2018}

\author{Svante Janson}
\thanks{Partly supported by the Knut and Alice Wallenberg Foundation}
\address{Department of Mathematics, Uppsala University, PO Box 480,
SE-751~06 Uppsala, Sweden}
\email{svante.janson@math.uu.se}
\newcommand\urladdrx[1]{{\urladdr{\def~{{\tiny$\sim$}}#1}}}
\urladdrx{http://www2.math.uu.se/~svante/}


\numberwithin{equation}{section}

\renewcommand\le{\leqslant}
\renewcommand\ge{\geqslant}

\allowdisplaybreaks

\theoremstyle{plain}
\newtheorem{theorem}{Theorem}[section]
\newtheorem{lemma}[theorem]{Lemma}

\newtheorem{corollary}[theorem]{Corollary}

\theoremstyle{definition}
\newtheorem{example}[theorem]{Example}
\newtheorem{definition}[theorem]{Definition}
\newtheorem{definitions}[theorem]{Definitions}
\newtheorem{problem}[theorem]{Problem}
\newtheorem{remark}[theorem]{Remark}

\newtheorem*{ack}{Acknowledgement}

\theoremstyle{remark}

\newenvironment{romenumerate}[1][-10pt]{
\addtolength{\leftmargini}{#1}\begin{enumerate}
 }{\end{enumerate}}

\newcounter{oldenumi}
{\setcounter{oldenumi}{\value{enumi}}
\begin{romenumerate} \setcounter{enumi}{\value{oldenumi}}}
{\end{romenumerate}}

\newcounter{thmenumerate}
\newenvironment{thmenumerate}
{\setcounter{thmenumerate}{0}%
 \def\item{\par
 \refstepcounter{thmenumerate}\textup{(\roman{thmenumerate})\enspace}}
}
{}

\newcounter{xenumerate}   

\newcommand\pfitemx[1]{\par#1:}
\newcommand\pfitemref[1]{\pfitemx{\ref{#1}}}

\newcommand{\refT}[1]{Theorem~\ref{#1}}

\newcommand{\refC}[1]{Corollary~\ref{#1}}

\newcommand{\refL}[1]{Lemma~\ref{#1}}

\newcommand{\refR}[1]{Remark~\ref{#1}}
\newcommand{\refRs}[1]{Remarks~\ref{#1}}
\newcommand{\refS}[1]{Section~\ref{#1}}
\newcommand{\refSs}[1]{Sections~\ref{#1}}
\newcommand{\refSS}[1]{Section~\ref{#1}}

\newcommand{\refD}[1]{Definition~\ref{#1}}
\newcommand{\refE}[1]{Example~\ref{#1}}
\newcommand{\refEs}[1]{Examples~\ref{#1}}

\newcommand{\refApp}[1]{Appendix~\ref{#1}}

\newcommand\XXX{XXX \marginal{XXX}}
\newcommand\REM[1]{{\raggedright\texttt{[#1]}\par\marginal{XXX}}}

\begingroup
  \divide\count255 by 60
  \multiply\count255 by -60
  \advance\count255 by \time
  \ifnum \count255 < 10 \xdef\klockan{\the\count1.0\the\count255}
  \else\xdef\klockan{\the\count1.\the\count255}\fi
\endgroup

\newcommand\nopf{\qed}   

\newcommand{\sumin}{\sum_{i=1}^n}

\newcommand{\sumkon}{\sum_{k=0}^n}
\newcommand{\sumkN}{\sum_{k=1}^N}
\newcommand{\sumkM}{\sum_{k=1}^M}
\newcommand{\sumkm}{\sum_{k=1}^m}
\newcommand{\prodin}{\prod_{i=1}^n}
\newcommand{\prodkm}{\prod_{k=1}^m}

\newcommand\set[1]{\ensuremath{\{#1\}}}
\newcommand\bigset[1]{\ensuremath{\bigl\{#1\bigr\}}}
\newcommand\Bigset[1]{\ensuremath{\Bigl\{#1\Bigr\}}}
\newcommand\biggset[1]{\ensuremath{\biggl\{#1\biggr\}}}
\newcommand\lrset[1]{\ensuremath{\left\{#1\right\}}}
\newcommand\xpar[1]{(#1)}
\newcommand\bigpar[1]{\bigl(#1\bigr)}
\newcommand\Bigpar[1]{\Bigl(#1\Bigr)}
\newcommand\biggpar[1]{\biggl(#1\biggr)}
\newcommand\lrpar[1]{\left(#1\right)}

\newcommand\xcpar[1]{\{#1\}}

\newcommand\bigabs[1]{\bigl|#1\bigr|}
\newcommand\Bigabs[1]{\Bigl|#1\Bigr|}

\def\rompar(#1){\textup(#1\textup)}    
\newcommand\xfrac[2]{#1/#2}

\newcommand\Bigparfrac[2]{\Bigpar{\frac{#1}{#2}}}

\def\xexp(#1){e^{#1}}
\newcommand\ceil[1]{\lceil#1\rceil}
\newcommand\floor[1]{\lfloor#1\rfloor}

\newcommand\setn{\set{1,\dots,n}}
\newcommand\setnn{[n]}
\newcommand\setNN{[N]}
\newcommand\ntoo{\ensuremath{{n\to\infty}}}
\newcommand\Ntoo{\ensuremath{{N\to\infty}}}

\newcommand\norm[1]{\|#1\|}
\newcommand\bignorm[1]{\bigl\|#1\bigr\|}
\newcommand\Bignorm[1]{\Bigl\|#1\Bigr\|}

\newcommand\punkt{.\spacefactor=1000}    
    
\newcommand\ie{i.e\punkt}
\newcommand\eg{e.g\punkt}

\newcommand\cf{cf\punkt}

\newcommand\ii{\mathrm{i}}

\newcommand\bbR{\mathbb R}
\newcommand\bbC{\mathbb C}
\newcommand\bbN{\mathbb N}

\newcommand\bbZ{\mathbb Z}

\newcounter{CC}
\newcounter{cc}

\renewcommand\Re{\operatorname{Re}}

\newcommand\E{\operatorname{\mathbb E{}}}
\renewcommand\P{\operatorname{\mathbb P{}}}

\newcommand\arccosh{\operatorname{arccosh}}

\newcommand\sign{\operatorname{sign}}

\newcommand\gd{\delta}

\newcommand\gf{\varphi}
\newcommand\gam{\gamma}

\newcommand\gl{\lambda}
\newcommand\gL{\Lambda}

\newcommand\gs{\sigma}

\newcommand\gth{\theta}
\newcommand\eps{\varepsilon}

\renewcommand\phi{\xxx}  

\newcommand\cD{\mathcal D}
\newcommand\cE{\mathcal E}
\newcommand\cF{\mathcal F}

\newcommand\cL{{\mathcal L}}
\newcommand\cM{\mathcal M}

\newcommand\cP{\mathcal P}

\newcommand\cS{{\mathcal S}}
\newcommand\cT{{\mathcal T}}

\newcommand\indic[1]{\boldsymbol1\xcpar{#1}} 
\newcommand\indicq[1]{\boldsymbol1_{\xcpar{#1}}}

\newcommand\smatrixx[1]{\left(\begin{smallmatrix}#1\end{smallmatrix}\right)}

\newcommand\qw{^{-1}}

\newcommand\qq{^{1/2}}

\newcommand\intoi{\int_0^1}
\newcommand\intoo{\int_0^\infty}

\newcommand\oi{[0,1]}
\newcommand\oio{[0,1)}
\newcommand\ooo{[0,\infty)}

\newcommand\dd{\,\mathrm{d}}
\newcommand\ddx{\mathrm{d}}

\newcommand\rhs{right-hand side}

\newcommand\xx{{\mathbf{x}}}
\newcommand\yy{\mathbf{y}}
\newcommand\tensor{\otimes}
\newcommand\hL{\hat L}
\newcommand\sss{\textsf{s}}
\newcommand\normpi[1]{\norm{#1}_{\pi}}
\newcommand\normpiq[2]{\norm{#1}_{\pi;\,\,#2}}
\newcommand\normpix[1]{\norm{#1}_{\pi}^*}
\newcommand\normpixq[2]{\norm{#1}_{\pi;\,\,#2}^*}
\newcommand\normpis[1]{\norm{#1}_{\pi,\sss}}
\newcommand\normpisq[2]{\norm{#1}_{\pi,\sss;\,#2}}
\newcommand\normpisx[1]{\norm{#1}_{\pi,\sss}^*}
\newcommand\normpip[1]{\norm{#1}_{\pi,+}}
\newcommand\normpipq[2]{\norm{#1}_{\pi,+;\,#2}}
\newcommand\normpipn[1]{\norm{#1}_{\pi,+,n}}

\newcommand\normpisp[1]{\norm{#1}_{\pi,\sss,+}}
\newcommand\normpispq[2]{\norm{#1}_{\pi,\sss,+;\,#2}}
\newcommand\normpispx[1]{\norm{#1}_{\pi,\sss,+}^*}
\newcommand\bignormpi[1]{\bignorm{#1}_{\pi}}

\newcommand\bignormpis[1]{\bignorm{#1}_{\pi,\sss}}

\newcommand\bignormpip[1]{\bignorm{#1}_{\pi,+}}

\newcommand\bignormpisp[1]{\bignorm{#1}_{\pi,\sss,+}}

\newcommand\Bignormpisp[1]{\Bignorm{#1}_{\pi,\sss,+}}
\newcommand\Bignormpispq[2]{\Bignorm{#1}_{\pi,\sss,+;\,#2}}

\newcommand\normabspi[1]{\norm{#1}_{|\pi|}}
\newcommand\normabspis[1]{\norm{#1}_{s,|\pi|}}
\newcommand\normp[1]{\norm{#1}_{+}}
\newcommand\normq[1]{\norm{#1}_{\Delta}}
\newcommand\normqp[1]{\norm{#1}_{\Delta,+}}
\newcommand\fL{\mathfrak L}
\newcommand\fP{\mathfrak P}
\newcommand\fS{\mathfrak S}
\newcommand\fSn{\fS_n}
\newcommand\nE{{}^n\!E}
\newcommand\nek{\nE;\bbK}
\newcommand\nef{\nE;F}
\newcommand\bbK{\mathbb K}
\newcommand\xxn{x_1,\dots,x_n}
\newcommand\fLs{\fL_\sss}
\newcommand\cLs{\cL_\sss}
\newcommand\tL{\tilde L}
\newcommand\hompol{homogeneous polynomial}
\newcommand\hompoln{\hompol{} of degree $n$}
\newcommand\hompolsn{\hompol{s} of degree $n$}
\newcommand\chp{\check p}
\newcommand\nn{^{\tensor n}}
\newcommand\nnx[1]{^{\tensor{#1}}}
\newcommand\nnpi{^{\tensor n}_{\pi}}
\newcommand\nnpip{^{\tensor n}_{\pi,+}}
\newcommand\nnpiq{^{\widehat{\tensor} n}_{\pi}}
\newcommand\snn{^{\vee n}}
\newcommand\snnx[1]{^{\vee #1}}
\newcommand\snnpi{^{\vee n}_{\pi}}
\newcommand\snnpix[1]{^{\vee #1}_{\pi}}
\newcommand\snnpip{^{\vee n}_{\pi,+}}
\newcommand\snnpipx[1]{^{\vee #1}_{\pi,+}}
\newcommand\snnpis{^{\vee n}_{\pi,\sss}}
\newcommand\snnpisx[1]{^{\vee #1}_{\pi,\sss}}
\newcommand\snnpisp{^{\vee n}_{\pi,\sss,+}}
\newcommand\snnpispx[1]{^{\vee #1}_{\pi,\sss,+}}
\newcommand\bL{\bar L}
\newcommand\tensors{\tensor\dotsm\tensor}
\newcommand\xxtensorn{x_1\tensors x_n}
\newcommand\vees{\vee\dotsm\vee}
\newcommand\xxveen{x_1\vees x_n}
\newcommand\ellaa{\ell_1^2}
\newcommand\ellaap{(\ellaa)_+}
\newcommand\pn{P_{\le n}}
\newcommand\chiab{\chi_{a,b}}
\newcommand\cd{[c,d]}
\newcommand\igs{\iota_\gs}
\newcommand\wii{[-1,1]}
\newcommand\kk{\kappa}

\newcommand\XX{{\mathbf{X}}}
\newcommand\muXX{\mu_{\XX}}

\newcommand\gfxx{\gf_{\xx}}
\newcommand\ff{\eta}
\newcommand\cMP[1]{\cM(\cP(#1))}
\newcommand\cMPS{\cMP{S}}
\newcommand\cl{c_{\textsf{L}}}
\newcommand\cs{c_{\sss}}
\newcommand\cp{c_{+}}
\newcommand\csp{c_{\sss,+}}
\newcommand\cssp{c_{\sss;\,\sss,+}}
\newcommand\cpsp{c_{+;\,\sss,+}}
\newcommand\Epp{E^+}
\newcommand\Ep{E_+}

\newcommand\kkx[1]{\kk(#1)}
\newcommand\kkn{\kkx{n}}
\newcommand\kknN{\kkx{n,N}}
\newcommand\xiu{u}
\newcommand\opii{[0,\frac{\pi}2]}
\newcommand\dual{^*}
\newcommand\exch{exchangeable}
\newcommand\ext{extendible}
\newcommand\Next{$N$-\ext}
\newcommand\PiNn{\Pi_{N,n}}
\newcommand\kknNm{\kk(n,N;m)}
\newcommand\citat[1]{``#1''} 

\begin{document}

\begin{abstract} 
We define new norms for symmetric tensors over ordered normed spaces; these
norms are defined by considering linear combinations of tensor products or
powers of positive elements only. Relations between the different norms are
studied. The results are applied to the problem of representing a finitely
exchangeable distribution as a mixture of powers, i.e, mixture of
distributions of i.i.d.\
sequences, using a signed mixing measure.
\end{abstract}

\maketitle


\section{Introduction}\label{S:intro}

Let $E$ be a normed space, and consider a tensor $\xx\in E\nn$.
By definition $\xx$ is a linear combination of elementary tensors
$x_1\tensors x_n$, and, roughly speaking, the projective tensor norm
$\normpi{\xx}$ measures 
how large such a linear combination has to be; see
the definition \eqref{normpi} below for a formal statement.

If the tensor $\xx$ is symmetric, it can also be written as a linear
combination of tensor powers $x\nn$. In general, such decompositions are
larger, and the symmetric projective norm \eqref{normpis} introduced by
\cite{Floret} 
measures how large. 

In the case when the normed space $E$ is an ordered space,
it also make sense to ask about decompositions into tensor products or
tensor powers of positive elements only.
We define in \eqref{normpip} and \eqref{normpisp} two norms on symmetric
tensors that measure the size of such decompositions. This gives four
different norms on the space $E\snn$ of symmetric tensors; 
they are all equivalent but, in general, different.

We study these norms and relations betweeen them in 
\refSs{Spos}--\ref{Sgamma}. 
In particular, we study the norms of the identity operator
between the four spaces obtained by equipping $E\snn$ with these norms, \ie,
the best constants in the inequalities relating these norms to each other.
(These constants depend on the space $E$ and on the order $n$.)
One of these constants is known as the \emph{polarization constant}
\cite{Dineen};
three other of them, defined in \refS{Spos}, are natural versions for 
ordered spaces, and we call them  
\emph{positive polarization constants}.
Among other results, we show that the space $\ell_1$ is extreme for several
of these polarization constants.

One motivation for the present paper is the problem of representing
finitely exchangeable distributions of random vectors as mixtures of 
distributions of independent
sequences. This problem is described more fully in \refS{Sexch}.
It is well known that, in contrast to de Finetti's theorem
for infinite exchangeable sequences,
such representations with a probability measure as the  mixing measure are
in general not possible for  finitely exchangeable distributions; however, a
substitute exists where the mixing measure is a signed measure
\cite{DM,Jaynes,KS06,SJ308}.
A natural question is how large the norm of this mixing measure has to be,
and it is shown in \refS{Sexch} that this is essentially equivalent to
studying one of the positive tensor norms defined in \refS{Spos}, in the
special case when $E=\ell_1$ (or a finite-dimensional $\ell_1^m$).
We use this to derive several new results on the optimal norm of the mixing
measure.

\refS{Sex} gives some simple explicit examples in the case when $E$ is a
Euclidean space.

\refSs{Spoly} and \ref{Stensor} contain background material, surveying
definitions and elementary properties of polarization, polarization
constants and tensor products. 
These sections provide background and easy references to
various facts for use in later sections. (There are no new results there.)

In the main part of the paper, starting with \refS{Spos}, we consider
ordered normed spaces, and thus spaces over $\bbR$.
However, in the introductory \refSs{Spoly} and \ref{Stensor},
no ordering is considered, so $E$ can be any normed space,
with real or complex scalars.

Another type of polarization constants, called \emph{linear polarization
  constants} has also been studied. There are, as far as we know, no direct
connections with the polarization constants studied here. 
However, we find it interesting to compare with these constants, and
therefore (and to prevent confusion with them),
we give a short survey of them in \refApp{Alinear}.

\subsection{Some notation}\label{SSnot}
We consider linear spaces over $\bbK$, where
$\bbK$ is either $\bbR$ or $\bbC$. 
In particular, $E$ or $F$ is always a normed space over $\bbK$. 
Furthermore,  $n\ge1$ is an integer, usually fixed but arbitrary.
(We sometimes tacitly assume that the spaces have non-zero dimension.)

For $1\le m<\infty$,
$\ell_p^m$ is $\bbK^m$ with the $\ell_p$-norm. 
In particular, $\ell_2^m$ is the usual Euclidean
space $\bbR^m$ or $\bbC^m$.
We also write $\ell_p^\infty=\ell_p$,
and let $\ell_p(S)$ denote the $\ell_p$ space with index set $S$, \ie,
$\ell_p(S):=L^p(S,\mu)$ where $\mu$ is the counting measure on $S$.
Thus $\ell_p=\ell_p(\bbN)$ and $\ell_p^m=\ell_p(\set{1,\dots,m})$.
The standard basis in $\ell_p$ or $\ell_p^m$ is denoted by $(e_i)$.

For a normed space $E$, 
$B(E):=\set{x\in E:\norm x\le1}$,
the closed unit ball of $E$.

``Positive'' should generally be interpreted as ``non-negative''.

For a real number $x$, $\floor x$ and $\ceil x$ are the integers obtained by
rounding $x$ downwards and upwards, respectively.

\section{Symmetric multilinear forms and polynomials}
\label{Spoly}

In this section, we review some basic theory of symmetric multilinear forms
and operators, including the important polarization formula.
See \eg{} \citet[Section 1.1]{Dineen} for further details.
In this section, we allow both real and complex scalars; we therefore denote
the scalar field by $\bbK$ ($=\bbR$ or $\bbC$).

$\fL(\nef)=\fL(E,\dots,E;F)$ denotes the space of all $n$-linear operators
$E^n\to F$. We will mainly consider the case $F=\bbK$:
$\fL(\nek)=\fL(E,\dots,E;\bbK)$ is the space of all $n$-linear forms
$E^n\to\bbK$. 

It is well-known that an $n$-linear operator $L:E^n\to F$ is continuous if and
only if it is bounded, \ie, if the norm
\begin{equation}\label{normL}
  \norm{L}:=\sup\bigset{|L(\xxn)|:\norm{x_1},\dots,\norm{x_n}\le1}
\end{equation}
is finite.
Let 
\begin{equation}
  \cL(\nef):=\bigset{L\in\fL(\nef):\norm{L}<\infty}
\end{equation}
be the space of bounded $n$-linear operators $E^n\to F$.
This is a normed space with the norm $\norm{\,}$ in \eqref{normL}.
(It is a Banach space if $F$ is complete, \eg{} if $F=\bbK$.)

\begin{definitions}
Let $\fSn$ be the symmetric group of the $n!$
permutations of $\set{1,\dots,n}$. 
  \begin{romenumerate}
  \item 
If $L\in\fL(\nef)$ and $n\in\fSn$, 
then $L_\gs\in\fL(\nef)$ is given by
\begin{equation}
  L_\gs(\xxn):=L\bigpar{ x_{\gs(1)},\dots,x_{\gs(n)}}.
\end{equation}

\item \label{Lsymm}
$L\in\fL(\nef)$ is \emph{symmetric} if $L_\gs=L$ for all $\gs\in\fSn$.
Let 
\begin{equation}
\fLs(\nef):=\set{L\in\fL(\nef):L\text{ is symmetric}}   
\end{equation}
be the space of
symmetric $n$-linear operators $E^n\to F$, 
and $\cLs(\nef):=\fLs(\nef)\cap\cL(\nef)$ the
subspace of bounded (or, equivalently, continuous) symmetric $n$-linear operators.

\item 
If $L\in\fL(\nef)$, then its \emph{symmetrization} $\tL\in\fLs(\nef)$ is
given by
\begin{equation}\label{tL}
  \tL:=\frac{1}{n!}\sum_{\gs\in\fSn} L_\gs.
\end{equation}
  \end{romenumerate}
\end{definitions}

Note that $L$ is symmetric $\iff L=\tL$, and that the symmetrization map
$L\mapsto\tL$ is a linear projection of 
$\fL(\nef)$ onto $\fLs(\nef)$ and 
of $\cL(\nef)$ onto $\cLs(\nef)$.

\subsection{Polynomials}

If $L:E^n\to F$ is an $n$-linear operator (or any function on $E^n$), we define
$\hL:E\to F$ by
\begin{equation}\label{Lddd}
  \hL(x):=L(x,\dots,x).
\end{equation}
In other words, $\hL$ is the restriction of $L$ to the diagonal.

\begin{definitions}
  \begin{romenumerate}
  \item 
A function $q:E\to K$ is a \emph{\hompoln} if $q=\hL$ for some $n$-linear
  form $L\in\fL(\nek)$.
Let
\begin{equation}
  \fP_n(E):=\set{\hL:L\in\fL(\nek)}
\end{equation}
be the space of all \hompoln{} on $E$.

\item 
If $p$ is a \hompol{} 
on $E$, let 
\begin{equation}
  \norm{p}:=\sup\bigset{|p(x)|:\norm{x}\le 1},
\end{equation}
\ie, the usual sup-norm of the restriction of $p$ to the unit ball of $E$.

\item 
Let
\begin{equation}
  \cP_n(E):=\set{p\in\fP_n(E):\norm{p}<\infty},
\end{equation}
the space of bounded \hompoln. 
(Here 'bounded' means bounded on the unit ball, as for linear forms.)
This is a normed space with the norm $\norm\,$; we shall see in \refC{CPB}
that it is a Banach space.
  \end{romenumerate}
\end{definitions}

\begin{remark}
  We can define general polynomials on $E$ as linear combinations of
  homogeneous polynomials of different degrees. We will not study general
  polynomials in the present paper. 
\end{remark}

Note that if $L\in\fL(\nek)$, 
then $\hL=\hat{\tL}$.
Hence, it suffices to consider symmetric $L$ to define \hompol{s}:
\begin{equation}
  \fP_n(E)=\set{\hL:L\in\fLs(\nek)}.
\end{equation}

\subsection{Polarization}
We have the following important \emph{polarization identity}.

\begin{lemma}
  \label{Lp}
If\/ $L\in\fL(\nef)$, then
\begin{equation}
  \label{polo}
\tL(\xxn)=\frac{1}{2^nn!}
\sum_{\eps_1,\dots,\eps_n=\pm1} \eps_1\dotsm\eps_n 
\,\hL\biggpar{\sumin \eps_ix_i}.
\end{equation}
In particular,
if\/ $L\in\fLs(\nef)$, then
\begin{equation}
  \label{pols}
L(\xxn)=\frac{1}{2^nn!}
\sum_{\eps_1,\dots,\eps_n=\pm1} \eps_1\dotsm\eps_n 
\,\hL\biggpar{\sumin \eps_ix_i}.
\end{equation}
\end{lemma}
\begin{proof}
  Expand
  \begin{equation}
	 \hL\biggpar{\sumin \eps_ix_i}
=\sum_{i_1,\dots,i_n=1}^n L\bigpar{\eps_{i_1}x_{i_1},\dots,\eps_{i_n}x_{i_n}}.
  \end{equation}
Thus
\begin{multline}
\sum_{\eps_1,\dots,\eps_n=\pm1} \eps_1\dotsm\eps_n 
 \hL\biggpar{\sumin \eps_ix_i}
\\
=\sum_{i_1,\dots,i_n=1}^n \sum_{\eps_1,\dots,\eps_n=\pm1} \eps_1\dotsm\eps_n 
L\bigpar{\eps_{i_1}x_{i_1},\dots,\eps_{i_n}x_{i_n}},  
\end{multline}
where the inner sum
vanishes unless $i_1,\dots,i_n$ is a permutation $\gs$ of $1,\dots,n$, in which
case it equals $2^n L_\gs(\xxn)$. Hence, \eqref{polo} follows by \eqref{tL},
and \eqref{pols} is a special case.
\end{proof}

\begin{remark}\label{Rpolgen}
  More generally, for any independent $\bbK$-valued random variables
  $\xi_1,\dots,\xi_n$ with finite $(n+1)$-th moments, $\E \xi_i=0$ and
  $\E|\xi_i|^2=1$, we have
\begin{equation}
  \label{polgen}
\tL(\xxn)
=\frac{1}{n!}\E \lrpar{\bar\xi_1\dotsm\bar\xi_n \hL\biggpar{\sumin \xi_ix_i}}.
\end{equation}
(The expectation in
\eqref{polgen} is well-defined since 
$\hL\bigpar{\sumin \xi_ix_i}$ lies in a finite-dimensional subspace of $F$
for any fixed $x_1,\dots,x_n$.)
The polarization identities \eqref{polo} and \eqref{pols} are 
obtained by taking $\xi_i=\pm1$ (with
probability $\frac12$ each). Sometimes, other choices are useful.
\end{remark}

\begin{corollary}
  \label{CP}
The mapping $\pi:L\mapsto\hL$ is a linear bijection of $\fLs(\nek)$ onto
$\fP_n(E)$. 
\nopf
\end{corollary}

If $p\in\fP_n(E)$, let $\chp$ denote $\pi\qw(p)$, \ie, the unique symmetric
$n$-linear form $\chp\in\fLs(\nek)$
such that $\hat{\chp}=p$. Thus $\chp$ is given by the \rhs{} of \eqref{pols},
with $\hL$ replaced by $p$.

\begin{lemma}
  Let $L\in\fLs(\nek)$. 
Then the following are equivalent.
\begin{romenumerate}
\item \label{lc}
$L:E^n\to\bbK$ is continuous.
\item \label{lb}
$L:E^n\to\bbK$ is bounded.
\item \label{hlc}
$\hL:E\to\bbK$ is continuous.
\item \label{hlb}
$\hL:E\to\bbK$ is bounded.
\end{romenumerate}
\end{lemma}
\begin{proof}
  \ref{lc}$\iff$\ref{lb} is well-known, as said above.

\ref{lc}$\implies$\ref{hlc} and \ref{lb}$\implies$\ref{hlb} 
are immediate consequences of the  definition  \eqref{Lddd}.

\ref{hlc}$\implies$\ref{lc} and \ref{hlb}$\implies$\ref{lb}  
follow by the polarization identity \eqref{pols}.
\end{proof}

Consequently, the space $\cP_n(E)$ defined above as the space of all bounded
\hompolsn{} is also the space of all continuous \hompolsn.

\begin{corollary}
  \label{CP2}
The bijection $\pi:\fLs(\nek)\to\fP_n(E)$ restricts to a bijection
$\cLs(\nek)\to\cP_n(E)$. 
\nopf
\end{corollary}

\begin{corollary}
  \label{CPB}
$\cP_n(E)$ is isomorphic to $\cLs(\nek)$ 
as normed spaces, \ie, with equivalence of norms.
Hence, $\cP_n(E)$  is a Banach space.
\nopf
\end{corollary}

More precisely, \eqref{Lddd} and \eqref{pols} yield the following
inequalities for $L\in\cLs(\nek)$ (or more generally $L\in\fLs(\nek)$,
allowing the values $+\infty$ for the norms).
\begin{align}
  \norm{\hL}&\le \norm{L}, \label{er1}
\\
\norm{L}& \le \frac{n^n}{n!} \norm{\hL}. \label{er2}
\end{align} 

Define, for a multilinear form (or any function) $L$ on $E^n$
\begin{equation}\label{normLs}
  \normq{L}:=\norm{\hL}=\sup\bigpar{|L(x,\dots,x)|:\norm x\le1}.
\end{equation}
Then \eqref{er1}--\eqref{er2} can also be written
\begin{align}
  \normq{L}&\le \norm{L}, \label{ma1}
\\
\norm{L} &\le \frac{n^n}{n!} \normq{L}. \label{ma2}
\end{align} 
Hence, $\norm\,$ and $\normq\,$ are two equivalent norms on $\cLs(\nek)$.

\subsection{Polarization constants}

\begin{definition}\label{Dc}
  The polarization constant $\cs(n,E)$ is defined by,
see \cite[Definition 1.40]{Dineen},
  \begin{equation}\label{c}
	\cs(n,E):=\sup_{L\in\cLs(\nek)}\frac{\norm{L}}{\norm{\hL}}
=\sup_{L\in\cLs(\nek)}\frac{\norm{L}}{\normq{L}},
  \end{equation}
where, as in similar suprema below, we 
define $\frac{0}0:=0$. 
Equivalently, $\cs(n,E)$ is the norm of the linear operator
$\pi\qw:\cP_n(E)\to\cLs(\nek)$, see \refC{CP2}.
\end{definition}

\begin{remark}
  It is an easy consequence of the Hahn--Banach theorem that the supremum
  \eqref{c} remains the same if we consider $n$-linear operators
  $L\in\cLs(\nef)$ where $F\neq0$ is a normed space.
\end{remark}

Since $\hL=\hat{\tL}$, we also have
  \begin{equation}\label{c2}
	\cs(n,E)=\sup_{L\in\cL(\nek)}\frac{\norm{\tL}}{\norm{\hat{\tL}}}
=\sup_{L\in\cL(\nek)}\frac{\norm{\tL}}{\norm{\hL}}.
  \end{equation}

By \eqref{c} and \eqref{er1}--\eqref{er2},
\begin{equation}
  \label{winston}
1\le \cs(n,E)\le \frac{n^n}{n!}.
\end{equation}

Both inequalities in \eqref{winston} can be attained.
(The upper bound in \eqref{winston} was conjectured by Mazur and Orlicz in
\citat{The Scottish Book}, and proved in 1932 by Martin; see \citet{Harris}
and the references there.)

\begin{example}
  \label{El2}
For any Hilbert space $H$ 
(real or complex; of finite or infinite dimension) and any $n\ge1$,
$\cs(n,H)=1$; see
\citet{Banach}. See further \cite{Harris}.
\end{example}

\begin{example}
  \label{El1}
For any $n\ge1$ and any $m\ge n$,
$\cs(n,\ell_1^m)=\cs(n,\ell_1)=n^n/n!$.

To see this, let $n\le m\le\infty$ 
and define
$L\in\cL({}^n\ell_1^m,\bbK)$ by
\begin{equation}\label{el1}
L(x_1,\dots,x_n)=\prodin x_{ii},
\qquad \text{where } x_i=(x_{ij})_{j=1}^m.
\end{equation}
$L$ is not symmetric, so we consider its symmetrization
$\tL$.
We have, letting $e_1,e_2,\dots$ be the usual basis vectors in $\ell_p^m$,
\begin{equation}
  \norm{\tL}\ge|\tL(e_1,\dots,e_n)|
=\frac{1}{n!}\sum_{\gs\in\fSn} L\bigpar{e_{\gs(1)},\dots,e_{\gs(n)}}
=\frac{1}{n!}
\end{equation}
and, by the arithmetic-geometric inequality,
if $x=(x_i)_1^m$,
\begin{equation}\label{eli}
|\hL(x)|=\prodin|x_i|
\le\Bigpar{\frac{1}{n}\sumin|x_i|}^n 
\le n^{-n}\norm{x}^n.
\end{equation}
Hence, $\norm{\hL}\le n^{-n}$, and by \eqref{c2},
\begin{equation}
  \cs(n,\ell_1^m)\ge \frac{\norm{\tL}}{\norm{\hL}}
\ge \frac{n^n}{n!}.
\end{equation}
The converse inequality follows by \eqref{winston}.

Consequently, recalling \eqref{winston} again, 
\begin{align}\label{gabriel}
  \sup_{E}\cs(n,E)=\cs(n,\ell_1^n)=\cs(n,\ell_1)
=\frac{n^n}{n!}.
\end{align}
Thus, $\ell_1$ is extremal among all normed spaces,
and so is $\ell_1^n$ when $n$ is given.
\end{example}

See \eg{} \cite{Dineen} and \cite{PappasKK}
for further examples.

\begin{remark}
  It seems likely that the polarization constants $\cs(n,E)$ (as well as other
  similar constants defined below) are (weakly) increasing in $n$, but as
  far as I know, this is an open problem. (Cf.\ \refR{Rsuper}.)
\end{remark}

\begin{remark}\label{Rc}
  \citet[Definition 1.10]{Dineen} defines also
  \begin{equation}\label{rc}
	\cs(E):=\limsup_\ntoo \cs(n,E)^{1/n}.
  \end{equation}
It is an obvious conjecture that
the limit always exists, \ie, that
$\limsup$ can be replaced by $\lim$ in \eqref{rc};
however, this seems to be unproven so far.
The same applies to the related quantities in \refR{Rcsp}.
Cf.\  \refR{Rsuper} for a positive result for another ``polarization constant''.

By \eqref{winston} and Stirling's formula, for any normed space $E$,
\begin{equation}\label{rc2}
  1\le \cs(E)\le e,
\end{equation}
with both bounds attained since $\cs(H)=1$ for a Hilbert space $H$ and
$\cs(\ell_1)=e$ by \refEs{El2} and \ref{El1}. 
As another example, \cite[Proposition 1.43]{Dineen} implies
that $\cs(\ell_\infty)\le e/2$.
\end{remark}

\section{Tensor products}\label{Stensor}

In this section we recall some basic properties of tensor products and
symmetric tensor products. (These results are not new, but we present them
in a form suitable for later use.)
See \eg{} \citet[Chapters 1 and 2]{Ryan}, 
\citet[Chapter 1]{Dineen} and \citet{Floret}
for basic definitions, further
details and many other things not mentioned here.
In particular, note that we only consider \emph{tensor powers}, \ie, tensor
products of a space with itself (one or several times).
Again, we allow in this section both real and complex scalars.

\subsection{The projective tensor norm}
Let $E\nn=E\tensor\dotsm\tensor E$ be the algebraic $n$:th tensor power of
$E$. Recall that an element $\xx\in E\nn$ can be written, non-uniquely, as a
linear combination 
\begin{equation}\label{tensor}
\xx=\sumkN a_k x_{1k}\tensor\dotsm\tensor x_{nk}  
\end{equation}
of \emph{elementary tensors} 
$x_{1k}\tensor\dotsm\tensor x_{nk} $
for some
$x_{ik}\in E$, $i=1,\dots,n$, $k=1,\dots,N$, and $a_k\in\bbK$.
(Here and below, $N$ is an arbitrary positive integer.)

The \emph{projective tensor norm} $\normpi{\,}$ on $E\nn$ is defined by
\begin{equation}\label{normpi}
  \normpi{\xx}:=
\inf\lrset{\sumkN |a_k|\norm{x_{1k}}\dotsm\norm{x_{nk}}
	:\xx=\sumkN a_k x_{1k}\tensor\dotsm\tensor x_{nk}  }.
\end{equation}
This is a norm on $E\nn$. 
We denote $E\nn$ with this norm by $E\nnpi$.

We use the notation $\normpiq{\xx}{E}$ when we want to show the space $E$
explicitly, but usually we omit $E$ from the notation.
(The same applies to the norms defined later.)

\begin{remark}
If $E$ has infinite dimension, then $E\nnpi$ is not complete even if $E$ is.
  The \emph{projective tensor power} $E\nnpiq$
of a Banach space $E$ is defined as
the completion of $E\nnpi$. 
The norms defined below on $E\nn$ or its subspace $E\snn$ (also defined
below)
are all equivalent
to $\normpi\,$, and thus the completions with respect to these norms
are the same, as vector spaces, as the completion $E\nnpiq$ or 
the corresponding completion of $E\snn$ (\ie, the closure of $E\snn$ in
$E\nnpiq$). 
Hence, the results below on \eg{} 
inequalities between the different norms extend
trivially to the completed spaces.

While it often is natural to work with completed spaces, we have in the
present paper not much need for them, and we will work with normed spaces
such as $E\nnpi$ without completing them.
Hence,
we leave extensions to completed tensor products to the reader.
\end{remark}

\begin{remark}
  \label{Rbop}
It is not difficult to see that for an elementary tensor $\xx=x_1\tensors x_n$,
\begin{equation}\label{bop}
  \normpi{\xxtensorn}=\norm{x_1}\dotsm\norm{x_n}.
\end{equation}
The projective norm is the largest norm on $E\nn$ that satisfies \eqref{bop}.
\end{remark}

\begin{remark}\label{R=}
Roughly speaking, the unit ball of $E\nnpi$ is spanned by the elementary
tensors 
$x_{1}\tensor\dotsm\tensor x_{n}$ with $x_1,\dots,x_n\in B(E)$.
More precisely $B(E\nnpi)$ equals the closed convex hull of
the set of these elementary tensors. 
If $\dim(E)<\infty$, we do not have to
take the closure
because the convex hull of a compact set is compact in a finite-dimensional
space \cite[Theorem 3.20(d)]{Rudin:FA};
thus $B(E\nnpi)$ then
equals the convex hull of the set of these elementary tensors.
This means that the infimum in \eqref{normpi} is attained when
$\dim(E)<\infty$. 
\end{remark}

\begin{remark}
  \label{Rx}
It follows from \eqref{normpi} 
or from \refR{R=}
that for any linear operator $T:E\nnpi\to F$,
where $F$ is a normed space,
\begin{equation}
  \label{rx}
\norm{T}=\sup\bigset{\norm{T(\xxtensorn)}:x_1,\dots,x_n\in B(E)}.
\end{equation}
Conversely, this characterizes $\normpi\,$.
\end{remark}

\begin{example}\label{Ematrix}
  In the finite-dimensional case $E=\bbK^m$ (with any norm),
the space $E\nnx2$ is naturally identified with the $m^2$-dimensional
space of $m\times m$ matrices. 
(We will use this without comment in some examples below.)
We recall two well-known examples of the projective tensor norm $\norm\,$ in
$E\nnx2$:  
If $E=\ell_1^m$, then the norm is the $\ell_1$-norm, so
$(\ell_1^m)\nnx2=\ell_1^{m\times m}$ \cite[Exercise 2.6]{Ryan}.
If $E=\ell_2^m$, then the norm in 
$(\ell_2^m)\nnx2$ of a matrix is its Trace class norm (also known as
nuclear norm and Schatten $S_1$ norm, 
see \eg{} \cite[\S3.8]{GK}, \cite[\S30.2]{Lax}, \cite[Chapter 48]{Treves}); 
if $A$ is a symmetric matrix 
(Hermitean in the complex case), 
then this norm equals
the sum of the absolute values of the $m$ eigenvalues.
\end{example}

The fundamental property of tensor products is that they linearize
multilinear operators. More precisely, in our case, for any linear space
$F$,
there is a natural
bijection between multilinear maps $L:E^n\to F$ and linear maps $\bL:E\nn\to
F$ 
determined by
\begin{equation}\label{LbL}
  L(x_1,\dots,x_n)=\bL(x_1\tensor\dotsm\tensor x_n).
\end{equation}
In particular, taking $F=\bbK$, this gives a $1$--$1$ correspondence between
$n$-linear forms on $E$ and linear forms on $E\nn$.
It follows  from \eqref{LbL}, 
 the definition \eqref{normL} and \eqref{rx}
that for an $n$-linear map
$L\in\fL(\nek)$,
the norm $\normpix{\bL}$ of $\bL$ as a linear functional on $E\nnpi$
equals the norm  $\norm{L}$ of $L$.

In the sequel, we abuse notation by denoting also the map
$E\nn\to \bbK$ corresponding to $L:E^n\to\bbK$ as in \eqref{LbL}
by the same symbol $L$ 
(instead of $\bL$). 
We thus have
\begin{equation}\label{normpix}
  \normpix L = \norm L.
\end{equation}
 The space $(E\nn)^*$
of bounded linear functionals on $E\nn$
is thus identified (isometrically) with $\cL(\nek)$.

\subsection{Symmetric tensor products}

A permutation $\gs\in\fSn$ defines an automorphism $\igs$ of $E\nn$
that is defined on elementary tensors by 
$\igs(x_1\tensors x_n):=x_{\gs(1)}\tensors x_{\gs(n)}$ and extended by linearity.
A tensor $\xx\in E\nn$ is \emph{symmetric} if $\igs(\xx)=\xx$ for every
$\gs\in\fSn$. 
The \emph{symmetric tensor product} $E\snn$ is the subspace
of $E\nn$ consisting of the symmetric tensors.

Define the \emph{symmetrization operator} 
$\gL:=\frac{1}{n!}\sum_{\gs\in\fSn} \igs$.
Then $\gL$ is a linear projection of $E\nn$ onto $E\snn$.
We define the \emph{elementary symmetric tensors}
\begin{equation}\label{vee}
  x_1\vee\dotsm\vee x_n:=
\gL(\xxtensorn)=
\frac{1}{n!}\sum_{\gs\in\fSn} 
x_{\gs(1)}\tensor\dotsm\tensor x_{\gs(n)}
\in E\snn.
\end{equation}
Note that the \emph{tensor powers} are elementary symmetric:
\begin{equation}
  \label{tp}
x\snn:=
x\vees x
=x\tensors x
=x\nn.
\end{equation}
We will mainly use the notation $x\nn$, also when discussing $E\snn$.

If $\xx\in E\snn$ is a symmetric tensor with a representation \eqref{tensor}, 
then also 
\begin{equation}\label{stensor}
\xx
=\gL(\xx)
=
\sumkN a_k x_{1k}\vees x_{nk}.  
\end{equation}
Hence, 
the linear space $E\snn$ is spanned by the tensors $x_1\vee\dots\vee x_n$.

Furthermore, $E\snn$ is also spanned by the (smaller) set of 
tensor powers $x\nn$ in \eqref{tp}.
This follows from the polarization identity \eqref{polo} applied to the
multilinear map 
$L:E^n\to E\nn$ given by
$L(x_1,\dots,x_n):= x_1\tensors x_n$,
which yields, using \eqref{vee} and \eqref{tL},
\begin{equation}\label{polt}
  \begin{split}
x_1\vees x_n &
=\tL(x_1,\dots,x_n)
  =
\frac{1}{2^nn!}
\sum_{\eps_1,\dots,\eps_n=\pm1} \eps_1\dotsm\eps_n 
\biggpar{\sumin \eps_ix_i}\nn.  	
  \end{split}
\end{equation}

It follows easily, using symmetrization by $\gL$ as in \eqref{stensor}, 
that for a symmetric tensor $\xx\in E\snn$, the projective
norm in \eqref{normpi} is also given by 
\begin{equation}\label{normpi2}
  \normpi{\xx}=
\inf\lrset{\sumkN |a_k|\norm{x_{1k}}\dotsm\norm{x_{nk}}
	:\xx=\sumkN a_k x_{1k}\vees x_{nk}  }.
\end{equation}

The 
\emph{symmetric projective tensor norm}
(or  \emph{projective s-tensor norm})
on $E\snn$,
introduced by \citet{Floret},
is defined by
\begin{equation}\label{normpis}
  \normpis{\xx}:=
\inf\lrset{\sumkN |a_k|\norm{x_{k}}^n
	:\xx=\sumkN a_k x_{k}\nn  }.
\end{equation}
By \eqref{normpi2}, \eqref{normpis} and \eqref{polt},
$\normpi{x}\le\normpis{x}<\infty$, so $\normpis\,$ is another norm on
$E\snn$. We will see in \eqref{normpipis} below that the norms are equivalent.
We denote the normed spaces obtained by equipping
$E\snn$ with the norms $\normpi\,$ and $\normpis\,$ by $E\snnpi$ and
$E\snnpis$, respectively.

\begin{remark}
  \label{Rbip}
It follows from \eqref{normpis} and \refR{Rbop} that for an elementary
tensor power 
$\xx=x\nn$,
\begin{equation}\label{bip}
  \normpis{x\nn}=
  \normpi{x\nn}=\norm{x}^n.
\end{equation}
The projective s-tensor norm is the largest norm on $E\snn$ that satisfies
\eqref{bip}. 
\end{remark}

\begin{remark}
  \label{R=s}
In analogy with \refR{R=},
the unit balls $B(E\snnpi)$  and $B(E\snnpis)$ 
equal the closed convex hull of
the sets
$\set{\xxveen:x_1,\dots,x_n\in B(E)}$
and $\set{\pm x\snn:x\in B(E)}$, respectively.
Again, 
if $\dim(E)<\infty$, we do not have to
take the closures, and thus
the infima in \eqref{normpi2} and \eqref{normpis} are attained.
\end{remark}

\begin{remark}
  \label{Rxs}
Similarly,
in analogy with \eqref{rx},
it follows from \eqref{normpi2} and \eqref{normpis} that
for any linear operator $T:E\snnpi\to F$,
where $F$ is a normed space,
\begin{align}
\norm{T}_{E\snnpi,F}&=\sup\bigset{\norm{T(\xxveen)}:x_1,\dots,x_n\in B(E)}.
\label{rix}
  \intertext{and}
\norm{T}_{E\snnpis,F}&=\sup\bigset{\norm{T(x\nn)}:x\in B(E)}.
\label{ris}
\end{align}
Conversely, these properties characterize the norms $\normpi\,$ and
$\normpis\,$ on $E\snn$.
\end{remark}

Similarly to the bijection between $\fL(\nE;F)$ and $\fL(E\nn;F)$ in
\eqref{LbL}, there is a bijection between symmetric multilinear maps $E^n\to
F$ and linear maps $E\snn\to F$ given by
\begin{equation}
  \label{LbLs}
L(x_1,\dots,x_n)=L(x_1\vees x_n),
\end{equation}
where we again abuse notation by using the same symbol for both operators.
In particular, taking $F=\bbK$, this yields a bijection between linear forms
on $E\snn$ and symmetric multilinear forms in $\fLs(\nek)$.

Let $L$ be a linear form on $E\snn$.
The norm of $L$ in the dual of $E\snnpi$ is
by \eqref{normpi2}, \eqref{LbLs} and \eqref{normL},
\begin{equation}\label{lt}
  \begin{split}
\normpix{L}
&=	
\sup\bigset{|L(x_1\vees x_n)|:\norm{x_1},\dots,\norm{x_n}\le1}
\\&
=	
\sup\bigset{|L(x_1,\dots, x_n)|:\norm{x_1},\dots,\norm{x_n}\le1}
\\&
=\norm{L}.
  \end{split}
\end{equation}
and the norm in the dual of $E\snnpis$ is
by \eqref{normpis}, \eqref{LbLs} and \eqref{normLs},
\begin{equation}\label{lts}
  \begin{split}
\normpisx{L}
&=	
\sup\bigset{|L(x\nn)|:\norm{x}\le1}
\\&
=	
\sup\bigset{|L(x,\dots, x)|:\norm{x}\le1}
\\&
=\normq{L}.
  \end{split}
\end{equation}

We obtain from \eqref{lt}--\eqref{lts} and the definition \eqref{c} 
immediately the following:

\begin{lemma}\label{LT}
  The polarization constant $\cs(n,E)$ is given by
  \begin{equation}
	\cs(n,E)=\sup_{L\in (E\snn)^*}\frac{\normpix L}{\normpisx L}.
  \end{equation}
In other words, $\cs(n,E)$ equals the norm of the identity map
$(E\snnpis)^*\to (E\snnpi)^*$.
\nopf
\end{lemma}

\begin{corollary}\label{CT}
The polarization constant 
$\cs(n,E)$ equals the norm of the identity map
$E\snnpi\to E\snnpis$.
In other words, for any $\xx\in E\snn$,
\begin{equation}
  \label{normpipis}
\normpi{\xx} \le \normpis{\xx}\le \cs(n,E)\normpi{ \xx}
\end{equation}
and $\cs(n,E)$ is the smallest constant for which this holds for all 
$\xx\in E\snn$. 
\end{corollary}
\begin{proof}
  \refL{LT} and duality.
\end{proof}

By \eqref{rix}, \refC{CT} is also equivalent to
\begin{equation}\label{johannes}
\cs(n,E)=\sup\bigset{\normpis{x_1\vees x_n}:\norm{x_1},\dots,\norm{x_n}\le1}.  
\end{equation}
In other words, by \eqref{normpis}, $\cs(n,E)$ describes how efficiently a
symmetric 
tensor $x_1\vees x_n$ with $x_1,\dots,x_n\in B(E)$ can be decomposed as a
linear combination of tensor powers $y_j\nn$.

\begin{example}\label{EBanach}
  For a Hilbert space $H$, \citet{Banach} showed $\cs(n,H)=1$, as said in
  \refE{El2}. Thus \refC{CT} yields
$\normpis{\xx}= \normpi{ \xx}$ for any $\xx\in H\snn$, $n\ge1$;
in other words, 
$H\snnpis=H\snnpi$ isometrically.
See  \cite[Section 5]{Friedland}.
\end{example}

\subsection{Functorial properties}\label{SSfunctor}
If $E$ and $F$ are two normed spaces and $T:E\to F$ is a bounded linear
operator, then $T$ induces a linear operator $T\nn:E\nn\to F\nn$ by
$T\nn(x_1\tensors x_n) = Tx_1\tensors Tx_n$; furthermore, $T\nn$ restricts
to $T\snn:E\snn\to F\snn$. 
We note the following well-known fact.

\begin{theorem}\label{Tfunctor}
 If $E$ and $F$ are  normed spaces
and $T:E\to F$ is a  bounded linear operator, then
$T\nn:E\nnpi\to F\nnpi$,
$T\snn:E\snnpi\to F\snnpi$
and $T\snn:E\snnpis\to F\snnpis$ all have norm $\norm{T}^n$.
\end{theorem}

\begin{proof}
  An immediate consequence of  \eqref{rx}, \eqref{rix}, \eqref{ris} 
together with \eqref{bop} and \eqref{bip}.
\end{proof}

There are some related simple results when we change the normed space.
Recall that the \emph{Banach--Mazur distance} between two isomorphic normed
spaces (in particular, Banach spaces) is $\inf\set{\norm {T}\norm{T\qw}}$,
taking the infimum over all isomorphisms $T:E\to F$.

\begin{theorem}
  \label{Tfunc}
  \begin{thmenumerate}
  \item 
If\/ $F$ is a quotient space of $E$, then $\cs(n,F)\le \cs(n,E)$.
\item 
If\/ $F$ is a $\kk$-complemented subspace of $E$, \ie, $F$ is a subspace and
there exists a projection $P:E\to F$ of norm $\norm{P}\le\kk$, then
$\cs(n,F)\le \kk^n \cs(n,E)$. In particular, if $F$ is $1$-complemented, then
$\cs(n,F)\le \cs(n,E)$.
\item 
If\/ $E$ and $F$ are isomorphic normed spaces, then
$\cs(n,F)\le d(E,F)^n \cs(n,E)$, where $d(E,F)$ is the Banach--Mazur distance.
In particular, $\cs(n,E)=\cs(n,F)$ when $E$ and $F$ are isometric.
  \end{thmenumerate}
\end{theorem}

The three parts of \refT{Tfunc} are proved by the same argument, which we
state more generally as a lemma.

\begin{lemma}\label{Lfunc}
  Suppose that $E$ and $F$ are normed spaces, that $T:E\to F$ is a bounded
  linear operator onto $F$, and that $A$ is a constant such that for every
  $x\in F$ and $\eps>0$, there exists $y\in E$ with $Ty=x$ and $\norm{y}\le
  (A+\eps)\norm{x}$. Then
$\cs(n,F)\le(A\norm T)^n\cs(n,E)$.
\end{lemma}

\begin{proof}
  We use \eqref{johannes}.
Let $x_1,\dots,x_n\in B(F)$ and let $\eps>0$.
By assumption, there exist $y_1,\dots,y_n\in E$ such that $\norm{y_i}\le
A+\eps$ and $Ty_i=x_i$. Then $T\snn(y_1\vees y_n)=x_1\vees x_n$ and thus,
using \eqref{johannes},
\begin{equation}
  \begin{split}
  \normpis{x_1\vees x_n}
&
\le \norm{T}^n   \normpis{y_1\vees y_n}
\le \norm{T}^n \cs(n,E)  \norm{y_1}\dotsm \norm{y_n}
\\&
\le \norm{T}^n \cs(n,E) (A+\eps)^n.	
  \end{split}
\raisetag\baselineskip
\end{equation}
Now let $\eps\to0$ and use \eqref{johannes} again.
\end{proof}

\begin{proof}[Proof of \refT{Tfunc}]
  We apply \refL{Lfunc} as follows:
  \begin{romenumerate}
  \item 
Let $T$ be the quotient mapping $E\to F$. Then $\norm{T}=1$, and, by
definition of the quotient norm, the assumption of the lemma holds with
$A=1$.
\item 
Let $T=P$ and let $A=1$; we can take $y=x$.
\item 
If $T:E\to F$ is an isomorphism, we take $y=T\qw x$ and the assumption holds
with $A=\norm{T\qw}$. Thus $\cs(n,F)\le(\norm{T}\norm{T\qw})^n\cs(n,E)$. Take
the infimum over $T$.
\qedhere
  \end{romenumerate}
\end{proof}

\begin{remark}\label{Rcfunc}
For the constants $\cs(E)$ defined in \refR{Rc} we obtain as an immediate
corollary of \refT{Tfunc} that in the three cases of the theorem, we have
$\cs(F)\le \cs(E)$, $\cs(F)\le \kk \cs(E)$ and $\cs(F)\le d(E,F)\cs(E)$,
respectively. 
\end{remark}

\begin{remark}\label{Rnotsubspace}
  It is not true in general  that $\cs(n,F)\le \cs(n,E)$ 
when $F$ is a subspace  of $E$.
For example, $\ell_1$
(as any separable Banach space) can be embedded isometrically 
as a subspace of $\ell_\infty$. 
However, by \cite[Proposition 1.43]{Dineen} and \eqref{gabriel},
$\cs(n,\ell_\infty)<\cs(n,\ell_1)$ for any $n\ge2$. 
(Also,
see \cite[p.\ 52]{Dineen}
and
\refR{Rc}, $\cs(\ell_\infty)<\cs(\ell_1)$.)
\end{remark}

\section{Positive tensor products and polarization constants}
\label{Spos}
In the remainder of the paper,
we assume that $\bbK=\bbR$, and that $E$ is an ordered
normed space, \ie, a normed space that is also an ordered linear space.
This means that there is given a closed
cone $\Epp$ of \emph{positive} elements in
$E$; the order is defined by $x\le y\iff y-x\in \Epp$, and, conversely,
$\Epp:=\set{x:x\ge0}$.

We assume also that $E=\Epp-\Epp$, \ie, that every $x\in E$ can be written as
a difference $y-z$ of two positive elements.
We define a new norm $\normp\,$ on $E$ by
\begin{equation}\label{normp}
  \normp{x}:=\inf\bigset{\norm{y}+\norm{z}:x=y-z,\; y\ge0,\; z\ge0}
\end{equation}
and note that the triangle inequality implies $\normp{x}\ge\norm x$.
Finally, we assume that
\begin{equation}\label{cp}
  \cp(E):=\sup\bigset{\normp x:\norm x\le1}
\end{equation}
is finite.
Thus, $\norm\,$ and $\normp\,$ are equivalent norms on $E$.
Let $\Ep$ denote $E$ equipped with the norm $\normp\,$.
Then, $\cp(E)$ is the norm of the identity operator $E\to\Ep$.

\begin{example}\label{Ecp}
Some standard examples are  $\ell_p^m$ and 
$\ell_p$, for $1\le p\le\infty$, and more generally  
$L^p(\cS,\cF,\mu)$ for any measure space $(\cS,\cF,\mu)$, with the standard
definition of positive elements.
It is easy to see that in these examples (for $m\ge2$) $\cp(E)=2^{1-1/p}$.
In particular, $\cp(\ell_1)=\cp(L^1(\cS,\cF,\mu))=1$, 
so in these spaces the norms
$\normp\,$ and $\norm\,$ coincide.
\end{example}

\begin{example}\label{Ecp2}
  The examples in \refE{Ecp} are examples of Banach lattices, which also
  include many other important Banach spaces, see \eg{} \cite{Schaefer} or
  \cite{Meyer-Nieberg} 
for  definition and further examples.
In a Banach lattice $E$, every $x\in E$ has a decomposition $x=x_+-x_-$ with 
$x_\pm\in \Epp$ and $\norm{x_\pm}\le\norm{x}$; hence $1\le\cp(E)\le2$.
\end{example}

\begin{example}\label{Ecp3}
  If we go beyond Banach lattices, then $\cp(E)$ may be arbitrarily large. A
  simple example is provided by $E=\bbR^2$ with usual positive cone (the first
  quadrant) and the norm $\norm{(x,y)}:=|x-y|+C|x+y|$ for a large constant
  $C$; then $\norm{(1,-1)}=2$ and $\normp{(1,-1)}=2C+2$. Thus $\cp(E)\ge  C+1$.
(In fact, equality holds.)
\end{example}

\subsection{Positive tensor products}

We are interested in decompositions of tensors using tensor products of
positive elements only.
If $E$ is an ordered normed space,  define 
in analogy with \eqref{normpi} and \eqref{normpis}
the tensor norms
\begin{equation}\label{normpip}
  \normpip{\xx}:=
\inf\lrset{\sumkN |a_k|\norm{x_{1k}}\dotsm\norm{x_{nk}}
	:\xx=\sumkN a_k x_{1k}\tensor\dotsm\tensor x_{nk},\, x_{ik}\ge0 }.
\end{equation}
on $E\nn$,  
and
\begin{equation}\label{normpisp}
  \normpisp{\xx}:=
\inf\lrset{\sumkN |a_k|\norm{x_{k}}^n
	:\xx=\sumkN a_k x_{k}\nn,\, x_k\ge0  }.
\end{equation}
on $E\snn$; these norms are thus defined using only positive elements in the
decompositions. 
For a symmetric tensor $\xx\in E\snn$, we have in analogy with
\eqref{normpi2}
also
\begin{equation}\label{normpip2}
  \normpip{\xx}=
\inf\lrset{\sumkN |a_k|\norm{x_{1k}}\dotsm\norm{x_{nk}}
	:\xx=\sumkN a_k x_{1k}\vees x_{nk},\, x_k\ge0 }.
\end{equation}
It is perhaps not obvious that $\normpisp \xx$ always is finite, \ie, that
there always exists a decomposition as in \eqref{normpisp}; this is part of
\refL{Lpi} below.

We first note that $\normpip\,$ is an ordinary projective tensor power norm, 
but for the  (in general) differently normed space $\Ep$.
 
\begin{lemma}\label{Lpi+}
The norm $\normpip\,$ equals the norm in
$(\Ep)_\pi\nn$.
\end{lemma}
\begin{proof}
Let (temporarily) $\normpipn\,$ denote the norm in $(\Ep)_\pi\nn$.

If $x\ge0$, then $\normp{x}=\norm{x}$. Hence \eqref{normpip} implies that 
$\normpipn{\xx}\le\normpip{\xx}$.

Conversely, 
it suffices to consider $\xx=x_1\tensors x_n$ with $x_1,\dots,x_n\in E$.
Let $\eps>0$, and choose $x_{i0},x_{i1}\in E$
such that $x_i=x_{i0}-x_{i1}$ and $\norm{x_{i0}}+\norm{x_{i1}}\le\normp{x_i}+\eps$,
see \eqref{normp}.
Then,
\begin{align}
\xx=
x_1\tensors x_n
=\sum_{j_1=0}^1\dotsm \sum_{j_n=0}^1(-1)^{\sum_i j_i}x_{1j_1}\tensors x_{nj_n}
\end{align}
and thus
\begin{align}
  \normpip{\xx}&
\le \sum_{j_1=0}^1\dotsm \sum_{j_n=0}^1
\norm{x_{1j_1}}\dotsm\norm{ x_{nj_n}}
=\prodin\bigpar{\norm{ x_{i0}}+\norm{ x_{i1}}}
\notag\\&
\le\prodin\bigpar{\normp{ x_{i}}+\eps}.
\end{align}
Letting $\eps\to0$ yields
$ \normpip{\xx}\le \prodin \normp{x_i}=\normpipn{\xx}$.
\end{proof}

\begin{remark}\label{Rpi+}
\refL{Lpi+} does not extend to the symmetric tensor products and norms. For an
example, let $E=\ell_1^2$, so $\Ep=E$ by \refE{Ecp}; however, 
$\normpis{(1,-1)\nnx2}=\norm{(1,-1)}^2=4$ by \eqref{bip}, while
$\normpisp{(1,-1)\nnx2}=8$ by \eqref{psi} and \eqref{olo} below.
\end{remark}

\begin{lemma}\label{Lpi}
  \begin{thmenumerate}
  \item \label{Lpip}
For every $\xx\in E\nn$,
\begin{equation}\label{lpip}
  \normpi \xx \le \normpip \xx \le \cp(E)^n \normpi \xx.
\end{equation}
\item \label{Lpisp}
There exists a constant $\gam(n)$ (not depending on $E$) such that
for every $\xx\in E\snn$,
\begin{equation}\label{lpisp}
  \normpis \xx \le \normpisp \xx 
\le \gam(n)\normpisq{ \xx}{\Ep}
\le \gam(n)\cp(E)^n \normpis \xx.
\end{equation}
  \end{thmenumerate}
\end{lemma}

\begin{proof}
  \pfitemref{Lpip}
The first inequality in \eqref{lpip} is trivial. 
Since the identity map $I:E\to\Ep$ has norm $\cp(E)$, 
the identity map $I\nn:E\nnpi\to(\Ep)\nnpi$ has norm $\cp(E)^n$, see
\refT{Tfunctor}, which yields the second inequality
by \refL{Lpi+}.

 \pfitemref{Lpisp}
Again, the first inequality is trivial. 
Furthermore, the argument just given for \ref{Lpip} shows also that
$I\snn:E\snnpis\to(\Ep)\snnpis$ has norm $\cp(E)^n$, 
which yields the third inequality in \eqref{lpisp}.

For the second inequality, 
by \eqref{normpis},
it suffices to consider a  tensor power
$\xx=x\nn$. Decompose $x=y-z$ with $y,z\ge0$.
Define, for $t\in\bbR$, the tensor $w(t)\in E\snn$ by
\begin{equation}\label{wt}
  w(t):=(y+tz)\nn=\bigpar{x+(1+t)z}\nn
=\sum_{i=0}^n \binom ni (t+1)^{n-i} x^{\vee i} \vee z^{\vee (n-i)},
\end{equation}
where we have used the binomial theorem in the commutative tensor algebra
$\bigcup_{n\ge0} E\snn$.
Note that $x\nn = w(-1)$, and that for $t\ge0$ we have $y+tz\ge0$ and thus
\begin{equation}
  \normpisp{w(t)} \le \norm{y+tz}^n\le (\norm{y}+t\norm{z})^n.
\end{equation}

Now suppose that $\mu$ is a finite signed measure on $\ooo$ such that
\begin{equation}\label{mu}
  \intoo (t+1)^j\dd\mu(t) =
  \begin{cases}
	1, & j=0,\\
0, & j=1,\dots,n.
  \end{cases}
\end{equation}
Then \eqref{wt} yields
\begin{equation}\label{intw}
  \intoo w(t)\dd\mu(t) = x\snn = x\nn.
\end{equation}
Suppose further that $\mu$ is supported at a finite number of points, \ie,
$\mu$ is a linear combination of Dirac measures $\sum_k \gl_k\gd_{t_k}$.
Then the integral in \eqref{intw} is a linear combination 
$\sum_k \gl_k w(t_k) = \sum_k \gl_k (y+t_kz)\nn$ and thus \eqref{normpisp}
yields
\begin{equation}
  \begin{split}
  \normpisp{x\nn}&
 \le \sum_k |\gl_k|\norm{y+t_kz}^n
\le \sum_k |\gl_k|(\norm{y}+t_k\norm{z})^n
\\&
\le \sum_k |\gl_k|\max(1,t_k)^n(\norm{y}+\norm{z})^n.	
  \end{split}
\end{equation}
Taking the infimum over all decompositions $x=y-z$ we obtain
\begin{equation}
  \begin{split}
  \normpisp{x\nn}
\le \sum_k |\gl_k|\max(1,t_k)^n\normp{x}^n
  \end{split}
\end{equation}
This implies the second inequality in \eqref{lpisp} with
\begin{equation}\label{mub}
  \gam(n)= \sum_k |\gl_k|\max(1,t_k)^n.  
\end{equation}

It remains to show that such a $\mu$ exists.
For this we choose $t_1<\dots<t_{n+1}$ arbitrarily in $\ooo$.
The equations \eqref{mu} become the system of linear equations
\begin{equation}\label{mua}
  \sum_{k=1}^{n+1} \gl_k (t_k+1)^j = 
  \begin{cases}
	1,&j=0,\\
0,&j=1,\dots,n.
  \end{cases}
\end{equation}
The coefficient matrix is the Vandermonde matrix with entries $(t_k+1)^j$,
$k=1,\dots,n+1$ and $j=0,\dots,n$; this matrix is non-singular and thus
\eqref{mua} has a solution.
\end{proof}

The decompositions used in the proof above are in general not optimal.
Optimal decompositions may be much harder to find; 
two non-trivial examples are given in \eqref{chit} and \eqref{chiq} with
\eqref{chiqq}.

From now on, we let $\gam(n)$ denote  the smallest possible constant such
that \eqref{lpisp} holds for all $E$ and all $\xx\in E\snn$.
We will show that $\gamma(n)=2^{n-1}$ in \refT{Tgamma}, but until this is
proved, we regard $\gamma(n)$ as an unknown constant.

\refL{Lpi} shows that for any normed space $E$, 
$\normpi\,$ and $\normpip\,$ are equivalent norms on $E\nn$,
and
$\normpis\,$ and $\normpisp\,$ are equivalent norms on $E\snn$.
We use $E\nnpip$, $E\snnpip$ and $E\snnpisp$ to denote 
$E\nn$ with the norm $\normpip\,$ and
$E\snn$ with the norms $\normpip\,$ and $\normpisp\,$, respectively.

\begin{remark}
  \label{Rbup}
In analogy with \eqref{bop} and \eqref{bip}, it follows that for 
a positive elementary tensor product $\xx=\xxtensorn$ with $x_1,\dots,x_n\ge0$,
\begin{equation}\label{bupp}
  \normpip{\xxtensorn}=
  \normpi{\xxtensorn}=
\norm{x_1}\dotsm\norm{x_m},
\end{equation}
and for a positive elementary tensor power
$\xx=x\nn$ with $x\ge0$,
\begin{equation}\label{bup}
  \normpisp{x\nn}=
  \normpis{x\nn}=
  \normpip{x\nn}=
  \normpi{x\nn}=\norm{x}^n.
\end{equation}
The norms $\normpip\,$ and $\normpisp\,$
are the largest norms on $E\nn$ and $E\snn$, respectively, that satisfy
\eqref{bupp} and \eqref{bup}.

However, note that (for $n\ge2$), \eqref{bupp} and
\eqref{bup} in general are false for
general $x\in E$; hence $\normpisp\,$ are not  tensor norms in the usual
sense. 
In fact, by \refL{Lpi+} and \eqref{bop} applied to $\Ep$,
\begin{align}\label{bup+}
  \normpip{\xxtensorn}=\prodin \normp{x_i},
\qquad x_1,\dots,x_n\in E.
\end{align}
Another counterexample for \eqref{bup}
is given by the same example $E=\ell_1^2$
and $\normpisp{(1,-1)\nnx2}=8$ as in \refR{Rpi+}, 
given by  \eqref{psi} and \eqref{olo} below; see also \eqref{csspn2}.
\end{remark}

\begin{remark}
  \label{R=+}
Let $B^+(E):=B(E)\cap \Epp$, the positive part of the unit ball.
In analogy with \refR{R=},
the unit balls $B(E\snnpip)$  and $B(E\snnpisp)$ 
equal the closed convex hull of
the sets
$\set{\pm\xxveen:x_1,\dots,x_n\in B^+(E)}$
and $\set{\pm x\snn:x\in B^+(E)}$, respectively.
Again, 
if $\dim(E)<\infty$, 
these equal the convex hulls (which already are closed);
hence,
the infima in \eqref{normpi2} and \eqref{normpis} are attained when
$\dim(E)<\infty$.
\end{remark}

\begin{remark}
  \label{Rxsp}
Similarly,
in analogy with \eqref{rix}--\eqref{ris},
it follows from \eqref{normpip2} and \eqref{normpisp} that
for any linear operator $T:E\snnpi\to F$,
where $F$ is a normed space,
\begin{align}
\norm{T}_{E\snnpip,F}&=\sup\bigset{\norm{T(\xxveen)}:x_1,\dots,x_n\in B^+(E)}.
\label{prix}
  \intertext{and}
\norm{T}_{E\snnpisp,F}&=\sup\bigset{\norm{T(x\nn)}:x\in B^+(E)}.
\label{pris}
\end{align}
Conversely, these properties characterize the norms $\normpip\,$ and
$\normpisp\,$ on $E\snn$.
\end{remark}

\begin{remark}\label{Rlattice}
Even if $E$ is a Banach lattice, $\normpip\,$ and $\normpisp\,$ are in
general not lattice norms, \ie, in general $|\xx|\le|\yy|$ does not imply
$\norm{\xx}\le\norm{\yy}$.
For example, 
consider (\cf{} \refRs{Rpi+} and \ref{Rbup})
$E=\ell_1^2$ and let $\xx=(1,-1)\nnx2=
\smatrixx{
 \phantom{-} 1 & -1 \\
-1 & \phantom{-}1}
\in E\nnx2$ and 
$\yy=|\xx|=\smatrixx{1&1\\1&1}=(1,1)\nnx2$.
Then, $|\xx|=\yy$ but, see \eqref{psi}, \eqref{olo} and \eqref{bup},
$\normpisp{\xx}=8$ and $\normpisp{\yy}=4$.

  For Banach lattices $E$ and $F$, \citet{Fremlin} defined a positive
  projective tensor norm $\normabspi\,$ on $E\tensor F$ such that the
  completion is a Banach lattice.
In particular, for a Banach lattice $E$, $\normabspi\,$ is defined on
$E\nn$, and there is also a symmetric version $\normabspis\,$ on $E\snn$,
inroduced by \citet{BuBuskes}.
It is easily seen that
\begin{align}
  \normabspi{\xx}&=\inf\bigset{\normpip{\yy}:\yy\ge|\xx|},
\\
  \normabspis{\xx}&=\inf\bigset{\normpisp{\yy}:\yy\ge|\xx|}.
\end{align}
\end{remark}

\begin{remark}\label{Rrank}
A related notion of \emph{non-negative rank} of a 
non-negative tensor $\xx$, 
meaning the smallest  $N$  in a decomposition \eqref{normpip} with
$a_k\ge0$,
has been studied by several authors,
see \eg{} \citet{QiEtal1,QiEtal2} and the references there.
Note, however, that we consider arbitrary $\xx$ above, and do not require
$a_k\ge0$.
\end{remark}

\subsection{Positive  polarization constants}

In analogy with \refC{CT},
we define $\csp(n,E)$, $\cssp(n,E)$,  and $\cpsp(n,E)$ 
as the norms of the identity map 
$E\snnpi\to E\snnpisp$, $E\snnpis\to E\snnpisp$, 
and $E\snnpip\to E\snnpisp$, respectively,
\ie,
\begin{align}
 \csp(n,E)&:=\sup_{\xx\in E\snn} \frac{\normpisp \xx}{\normpi \xx},\label{cspn}
\\
\cssp(n,E)&:=\sup_{\xx\in E\snn} \frac{\normpisp \xx}{\normpis \xx},\label{csspn}
\\
\cpsp(n,E)&:=  \sup_{\xx\in E\snn}\frac{\normpisp{\xx}}{\normpip{\xx}}.\label{cpspn}
\end{align}

By  \eqref{rix}, \eqref{ris} and \eqref{normpip2}, it suffices to
consider elementary tensors $\xx=x_1\vees x_n$ 
in \eqref{cspn} and \eqref{cpspn}
and $\xx=x\nn$ in \eqref{csspn}, \ie,
\begin{align}
  \csp(n,E)&=\sup_{x_1,\dots,x_n\in E} 
  \frac{\normpisp{x_1\vees x_n}}{\norm{x_1}\dotsm\norm{x_n}}
,\label{cspn2}
\\
  \cssp(n,E)&=\sup_{x\in E} \frac{\normpisp {x\nn}}{\norm{x}^n}. \label{csspn2}
\\
  \cpsp(n,E)&=\sup_{x_1,\dots,x_n\in \Epp} 
  \frac{\normpisp{x_1\vees x_n}}{\norm{x_1}\dotsm\norm{x_n}}
\end{align}
Since $\norm{x}=\normp{x}$ when $x\ge0$, it follows from
\eqref{normpisp} that 
\begin{align}
  \label{normpisp+}
\normpispq\xx{E}=\normpispq\xx{\Ep}, 
\end{align}
and thus \refL{Lpi+} 
implies
\begin{align}
\cpsp(n,E)
=\csp(n,\Ep).
\end{align}
We will therefore usually ignore $\cpsp$, and leave it to the reader.

We may also consider
the identity maps $E\nnpi\to E\nnpip$
and $E\snnpi\to E\snnpip$, but we then do not need any new notation 
since it was shown in the proof of \refL{Lpi} that
both have norm $\cp(E)^n$, \ie,
\begin{align}\label{cpip}
  \sup_{\xx\in E\nn}\frac{\normpip{\xx}}{\normpi{\xx}}
=
  \sup_{\xx\in E\snn}\frac{\normpip{\xx}}{\normpi{\xx}}
=\cp(E)^n.
\end{align}

Note also that the inverses of all identity maps considered here have norm 1.
Thus, or directly from the definitions, 
$\csp(n,E)\ge1$,
$\cssp(n,E)\ge1$,
$\cpsp(n,E)\ge1$,
and $\cp(E)\ge1$.

Several inequalities between the different polarization constants
follow directly from the definitions and \refC{CT},
by considering compositions of the identity maps. For example,
\begin{equation}\label{abc}
\max\bigpar{\cs(n,E),
\cssp(n,E)}\le
  \csp(n,E)\le \cs(n,E)\cssp(n,E).
\end{equation}
Similarly,
by \eqref{csspn} and \eqref{lpisp},
\begin{equation}\label{csspgam}
1\le  \cssp(n,E)\le \gam(n)\cp(E)^n.
\end{equation}
Moreover, using \eqref{normpisp+},
$\gam(n)$ is the smallest constant such that 
$\normpispq\xx{\Ep} 
\le \gam(n) \normpisq\xx{\Ep}$
for all normed spaces $E$ and all $\xx\in E\snn$, i.e.,
\begin{equation}\label{gamoa}
  \gam(n)=\sup_E{\cssp(n,\Ep)}.
\end{equation}
Using \eqref{csspgam}, we thus also have
\begin{equation}\label{gamob}
  \gam(n)=\sup_E\frac{\cssp(n,E)}{\cp(E)^n}.
\end{equation}

\begin{example}\label{Edeka}
  For $n=1$, $E^{\vee1}=E^{\tensor1}=E$. 
It is obvious that 
the norms $\normpi x=\normpis x =\norm x$
for any $x\in E$; furthermore, see \eqref{normpisp} and \eqref{normp},
$\normpisp x=\normp x$. 
In particular, by \eqref{cspn}--\eqref{csspn} and \eqref{cp},
\begin{equation}\label{edeka}
\csp(1,E)=  \cssp(1,E)=\cp(E).
\end{equation}
Thus $\gam(1)=1$.
\end{example}

We note also that the definitions \eqref{normpip}--\eqref{normpisp}
and \eqref{bup+} 
imply
$
\normpisp{x\nn}
\ge
\normpip{x\nn}
=\normp{x}^n.
$
Thus,
using \eqref{abc},
\eqref{csspn2},
and \eqref{cp},
\begin{align}\label{lisken}
\csp(n,E)\ge
  \cssp(n,E)\ge \cp(E)^n.
\end{align}

\begin{example}
  \label{EH}
If $H$ is a Hilbert space $H$,
then \eqref{abc} and  \refE{EBanach}
yield
\begin{align}  \label{abcH}
  \cssp(n,H)=\csp(n,H).
\end{align}

We will see in \refEs{El2+} and \ref{El2g} that 
the result by \citet{Banach} in
\refE{EBanach} does not extend to the positive tensor norms, \ie, in general
$\normpipq{\xx}{H}\neq \normpispq{\xx}{H}$, even when $H$ is $\ell_2^2$ with
the usual ordering.

Furthermore, \refE{El2g} also shows that for this example
$\csp(2,\ell_2^2)=
  \cssp(2,\ell_2^2)=3$,
and thus the second inequality in \eqref{lisken} is strict; recall that
$\cp(\ell_2^2)^2=2$ by \refE{Ecp}.

\end{example}

\subsection{Multilinear forms on ordered spaces}\label{SSmulti+}
If $E$ is an ordered normed space, define for an $n$-linear form
$L\in\fL(\nek)$, in analogy with \eqref{normL} and \eqref{normLs},
\begin{align}
  \normp{L}&:=\sup\bigset{|L(x_1,\dots,x_n)|:
     \norm{x_1}=\dots=\norm{x_n}\le1,\, x_1,\dots,x_n\ge0 }.
\label{normL+}
\\
\normqp{L}&:=\normp{\hL}=\sup\bigset{|L(x,\dots,x)|:\norm x\le1,\, x\ge0}.
\label{normLs+}
\end{align}
Then $\normp{L}$ equals the norm $\normpip{L}^*$ in the dual of $E\nnpip$.
If $L$ is symmetric, then also
$\normp{L}$ equals the norm $\normpip{L}^*$ in the dual of $E\snnpip$,
and
$\normqp{L}$ equals the norm $\normpisp{L}^*$ in the dual of $E\snnpisp$.

By duality, $\csp(n,E)$ and $\cssp(n,E)$ equal the norms of the identity
operators $(E\snnpisp)^*\to(E\snnpi)^*$ and
$(E\snnpisp)^*\to(E\snnpis)^*$, respectively.
Hence, using \eqref{lt}--\eqref{lts},
\begin{align}
  \csp(n,E)&=\sup_{L\in\cLs(\nek)}\frac{\norm L}{\normqp L},
\\
  \cssp(n,E)&=\sup_{L\in\cLs(\nek)}\frac{\normq L}{\normqp L}.
\end{align}

\subsection{Functorial properties}\label{SSfunctor+}

We have functorial properties similar to the ones in \refT{Tfunctor}, but
now only for positive operators.

\begin{theorem}\label{Tfunctor+}
 If $E$ and $F$ are ordered normed spaces
and $T:E\to F$ is a positive bounded linear operator, then
$T\nn:E\nnpip\to F\nnpip$,
$T\snn:E\snnpip\to F\snnpip$
and $T\snn:E\snnpisp\to F\snnpisp$ all have norm $\norm{T}^n$.  
\end{theorem}

\begin{proof}
  An immediate consequence of the definitions 
\eqref{normpip} and \eqref{normpisp}
together with \eqref{bup}.
\end{proof}

It follows that there is a version of \refT{Tfunc} for $\csp(n,E)$ and
$\cssp(n,E)$, but more restrictive; for example, the analogue of  (ii) holds
provided we assume that the injection $F\to E$ and the projection $P$ are
positive operators; similarly, the analogue of (iii) holds if we define an
``ordered Banach--Mazur distance'' between isomorphic ordered normed spaces by
considering only order isomorpisms $T:E\to F$.

\section{$\ell_1$ is extreme}\label{Sell1}

We have seen in \refE{El1} that $\ell_1$ and $\ell_1^n$ are extremal among
all normed spaces 
for $\cs(n,E)$. 
The next two theorems show that they are extremal also for $\csp$ and
$\cssp$, provided we compensate for $\cp(E)$; recall that \refE{Ecp3} shows
that $\cp(E)$ may be arbitrarily large, so \eqref{edeka} 
and \eqref{lisken} show that
$\sup_E \csp(n,E)=  \sup_E\cssp(n,E)=\infty$ for any $n\ge1$.

\begin{theorem}\label{Tcsp}
If\/ $n\le m\le \infty$, then
$\csp(n,\ell_1^m)=\kkn$, where
\begin{align}\label{tcsp1}
 \kkn:=
\normpispq{e_1\vees e_n}{\ell_1}=\normpispq{e_1\vees e_n}{\ell_1^n}.
\end{align}
Furthermore,
\begin{align}\label{tcsp2}
\sup_E\frac{\csp(n,E)}{\cp(E)^n}
=
\sup_{E:\,\cp(E)=1}\csp(n,E)
=\csp(n,\ell_1)
=\csp(n,\ell_1^n)
=\kkn
.
\end{align}
\end{theorem}

\begin{proof}
First, note that the natural injection $\ell_1^n\to\ell_1$ and projection
$\ell_1\to\ell_1^n$ have norm 1, and that this implies the equality of the
two tensor norms in \eqref{tcsp1} by \refT{Tfunctor+}.

Let $x_1,\dots,x_n\in E$ with $x_i\ge0$ and $\norm{x_i}=1$.
Define a linear map $T:\ell_1^n\to E$ by $T e_i:=x_i$. Then $T$ is positive
and $\norm{T}=1$, and thus, by \refT{Tfunctor+},
$T\snn:(\ell_1^n)\snnpisp\to E\snnpisp$ has norm 1. Hence,
\begin{align}
  \normpispq{\xxveen} {E}&
=\normpispq{T\snn (e_1\vees e_n)}{E}
\notag\\&
\le\normpispq{e_1\vees e_n}{\ell_1^n}
=\kkn.
\end{align}
It now follows from \eqref{normpip2} that for any $\xx\in E\snn$, 
\begin{align}\label{eva}
  \normpisp{\xx}
\le\kkn \normpip{\xx}.
\end{align}
Combining \eqref{eva} and \eqref{lpip} yields
$  \normpisp{\xx}
\le\kkn \normpip{\xx}
\le\kkn \cp(E)^n\normpi{\xx}
$  
and thus
\begin{align}\label{bryngel}
  \csp(n,E)\le\kkn\cp(E)^n.
\end{align}
It follows immediately from \eqref{bryngel} and $\cp(\ell_1^m)=1$ that 
$\csp(n,\ell_1^m)$ and
all terms in \eqref{tcsp2} are at most $\kkn$.

Conversely, if $n\le m\le\infty$, then, 
using the injection $\ell_1^m\to\ell_1$,
\begin{align}\label{ellika}
  \kkn
&=
\normpispq{e_1\vees e_n}{\ell_1}
\le
\normpispq{e_1\vees e_n}{\ell_1^m}
\notag\\&
\le
\csp\xpar{\ell_1^m}\normpiq{e_1\vees e_n}{\ell_1^m}
=\csp\xpar{\ell_1^m}.
\end{align}
Hence $\csp\xpar{\ell_1^m}=\kkn$.
Furthermore, \eqref{ellika} implies that
each term in \eqref{tcsp2} is at least $\kkn$, so equalities holds.
\end{proof}

\begin{example}
  \label{Ekk2}
We show that $\kk(2)=3$.
This can be shown using the general results 
\eqref{uv} and \eqref{mumrik1} 
in \refR{Rbetter}
and
\refSS{SSbinary},
but we give a direct proof.

For an upper bound, we use the decomposition 
\begin{align}
  e_1\vee e_2 = 2\xpar{\tfrac12 e_1+\tfrac12 e_2}\nnx2
-\tfrac12 e_1\nnx2-\tfrac12 e_2\nnx2.
\end{align}

For a lower bound, we consider the linear map $L:(\ell_1^2)\snnx2\to\bbR$
given by 
$e_1\dual\tensor e_1\dual+e_2\dual\tensor e_2\dual-6e_1\dual\vee e_2\dual$,
\ie,
$\smatrixx{a&b\\b&c}\mapsto a+c-6b$. A positive unit vector in $\ell_1^2$ is
$(x,1-x)$ for some $x\in\oi$, and
\begin{align}
  L\bigpar{(x,1-x)\nnx2}=x^2+(1-x)^2-6x(1-x)=1-8x(1-x).
\end{align}
Since $0\le x(1-x)\le\frac{1}{4}$, 
$|L\bigpar{(x,1-x)\nnx2}|\le 1$, and thus, by \eqref{pris},
$\normpispx{L}\le1$. Furthermore, $L(e_1\vee e_2)=-3$. Hence
$\normpisp{e_1\vee e_2}\ge3$.

Consequently,
\begin{align}\label{ekk2}
  \kk(2)=\normpispq{e_1\vee e_2}{\ell_1}=3.
\end{align}
\end{example}

We study the constant $\kkn$ further in \refS{Sexch}, where it plays an
important role.

\begin{theorem}\label{Tcssp}
  If\/ $2\le m\le \infty$, then
$\cssp(n,\ell_1^m)=\gam(n)$.
Thus,
\begin{align}\label{tcssp}
\sup_E\frac{\cssp(n,E)}{\cp(E)^n}
=
\sup_{E:\,\cp(E)=1}\cssp(n,E)
=\cssp(n,\ell_1)
=\cssp(n,\ell_1^2)
=\gam(n).
\end{align}
\end{theorem}
We will find the explicit value $2^{n-1}$ in \refT{Tgamma}.

\begin{proof}
  Since $\cp(\ell_1^m)=1$, 
$\cssp(n,\ell_1^m)\le\gam(n)$ by \eqref{csspgam}.

Conversely, suppose that $x=y-z$ with $y,z\in \Epp$.
Let $y_0:=y/\norm{y}$ and $z_0:=z/\norm{z}$ 
(with $0/0:=0$).
Further, assuming $m\ge2$, let 
$\xiu:=(\norm{y},-\norm z,0,\dots)\in\ell_1^m$; then $\norm{\xiu}=\norm y+\norm
z$.

Define the linear map $T:\ell_1^m\to E$ 
by
$T(a_1,a_2,\dots)=a_1 y_0+a_2 z_0$. Then $T(\xiu)=y-z=x$.
Furthermore, 
$T$ has norm (at most) 1 and maps positive
elements to positive, and therefore by \refT{Tfunctor+},
$T\nn$ maps $(\ell_1^m)\snnpisp$ into $E\snnpisp$ with norm at most 1.
Consequently,
recalling \eqref{bip},
\begin{align}
    \normpispq{x\nn}{E}
&=
  \normpispq{T\nn\xiu\nn}{E}
\le   \normpispq{\xiu\nn}{\ell_1^m}
\le \cssp(n,\ell_1^m)  \normpisq{\xiu\nn}{\ell_1^m}
\notag\\&
= \cssp(n,\ell_1^m)  \norm{\xiu}^n
= \cssp(n,\ell_1^m)  (\norm{y}+\norm{z})^n.	
\label{kroken}
\end{align}
Taking the infimum over all decompositions $x=y-z$ with $y,z\in \Epp$ yields
\begin{equation}
  \normpisp{x\nn}
\le \cssp(n,\ell_1^m)  \normp{x}^n
.
\end{equation}
This holds for every $x\in E$, and hence, by
\eqref{csspn2} and \eqref{normpisp+}, 
\begin{equation}
\cssp(n,\Ep)\le
 \cssp(n,\ell_1^m).
\end{equation}
This holds for every normed space $E$, and thus \eqref{gamoa} shows
$\gam(n)\le\cssp(n,\ell_1^m)$. 
Hence, each term in \eqref{tcssp} is at least $\gam(n)$.
On the other hand, 
$\cssp(n,\ell_1^m)$ and all
terms in \eqref{tcssp} are at most $\gam(n)$ by
\eqref{csspgam}. Hence, equalities hold.
\end{proof}

\begin{remark}\label{Rcsp}
  In analogy with \eqref{rc}, we can define
  \begin{align}
	\csp(E)&:=\limsup_\ntoo \csp(n,E)^{1/n},\label{cspoo}
\\
	\cssp(E)&:=\limsup_\ntoo \cssp(n,E)^{1/n}.\label{csspoo}
  \end{align}
By \eqref{abc}, \eqref{rc2}, \eqref{csspgam} and \refT{Tgamma} below,
\begin{align}
1\le  \cssp(E)&\le\csp(E)\le\cs(E)\cssp(E)\le e\cssp(E), \label{elea}
\\
 \cssp(E)&\le 2\cp(E). \label{eleb}
\end{align}

For example, by 
\refT{Tgamma}, $\cssp(\ell_1^m)=2$ for
$2\le m\le\infty$. By \eqref{elea}, $2\le \csp(\ell_1^m)\le2e$; we do not
know the exact value.
\end{remark}

\section{The value of $\gamma(n)$}\label{Sgamma}

The proof of \refL{Lpi} yields an upper bound for $\gam(n)$ in
\eqref{mub}--\eqref{mua}. However, it seems difficult to evaluate this
exactly in general, and we do not know whether this method yields an upper
bound is optimal. We thus find $\gamma(n)$ by a different method, using
\refT{Tcssp}. 
(This gives another proof of \refL{Lpi}\ref{Lpisp}.)

\begin{theorem}
  \label{Tgamma}
  \begin{romenumerate}
  \item 
For $n\ge1$, $\gam(n)=2^{n-1}$.
\item \label{Tgammacx}
  If\/ $2\le m\le \infty$ and $n\ge1$, then
$\cssp(n,\ell_1^m)=2^{n-1}$.
  \end{romenumerate}
\end{theorem}

\begin{proof}
By \refT{Tcssp}, 
$\gam(n)=\cssp(n,\ell_1^m)$, for any $m\ge2$.
Hence, the two parts are equivalent,
and it suffices to prove \ref{Tgammacx} with $m=2$. 
Thus, let $E=\ell_1^2$ and
use \eqref{csspn2}, which yields
\begin{equation}\label{gamab}
  \gam(n)=\cssp(n,\ell_1^2)
=\sup_{a,b\in\bbR} \frac{\normpisp {(a,b)\nn}}{(|a|+|b|)^n}.
\end{equation}
Fix $n\ge1$ and write, for convenience, 
\begin{equation}\label{psi}
\psi(a,b):=\normpispq {(a,b)\nn}{\ell_1^2}.    
\end{equation}
Since $-(a,b)=(-a,-b)$, it  suffices to consider $a\ge0$.
Obviously, if  $a,b\ge0$, then $(a,b)\in \ellaap$ and thus, 
by \eqref{bup},
\begin{equation}\label{psi++}
  \psi(a,b):=\norm{(a,b)}^n=(a+b)^n,
\qquad a,b\ge0.
\end{equation}
Hence, the interesting case is $a>0>b$.
However, we continue to consider general $a,b\in\bbR$.

The unit vectors in $\ellaap$ are $(x,1-x)$, $x\in\oi$.
Consequently, the definition \eqref{normpisp}
can be written as
\begin{equation}\label{neutrino}
\psi(a,b):=
  \normpisp{(a,b)\nn}
:=
\inf\norm{\mu}=\inf\intoi|\ddx\mu|(x),
\end{equation}
taking the infimum over all signed measures of the type 
$\mu=\sumkN a_k\gd_{x_k}$ on $\oi$ such that
\begin{equation}
  \label{ab}
\intoi (x,1-x)\nn\dd\mu(x)=(a,b)\nn.
\end{equation}
In other words, we take the infimum over all signed measures with
finite support in $\oi$ that satisfy \eqref{ab}. On the other hand, for any
signed measure on $\oi$, 
\begin{equation}
  \Bignormpisp{\intoi (x,1-x)\nn\dd\mu(x)}
\le
\intoi  \bignormpisp{(x,1-x)\nn}|\dd\mu|(x)
=
\intoi|\dd\mu|(x)
\end{equation}
since the integral exists as a Bochner integral in $(\ellaa)\snnpisp$.
(Recall that the spaces are finite-dimensional, so there is no problem with
convergence.) 
Consequently, we can just as well take the infima in \eqref{neutrino} over
all signed measures $\mu$ on $\oi$ satisfying \eqref{ab}.

Expanding the tensor products in \eqref{ab} in $(\ellaa)\nn$, we see that
\eqref{ab} is equivalent to the system of equations
\begin{equation}
  \label{abk}
\intoi x^{n-k}(1-x)^k\dd\mu(x)=a^{n-k}b^k,
\qquad k=0,\dots,n.
\end{equation}

The coefficients of the $n+1$ polynomials $q_k(x):=x^{n-k}(1-x)^k$,
$k=0,\dots,n$, form a
triangular matrix which is non-singular; consequently these polynomials form
a basis in the $(n+1)$-dimensional space $\pn$ of polynomials (of a real
variable) of degree at most $n$.
Hence, there exists a unique linear functional $\chiab$ on $\pn$ such that
\begin{equation}
  \label{chiab}
\chiab(q_k)=a^{n-k}b^k,\qquad k=0,\dots,n, 
\end{equation}
and \eqref{abk} is equivalent to
$\intoi q_k(x)\dd\mu(x)=\chiab(q_k)$, $k=0,\dots,n$, and thus to
\begin{equation}\label{abp}
  \intoi p(x)\dd\mu(x)=\chiab(p),
\qquad p\in\pn.
\end{equation}
For a compact interval $[c,d]\subset \bbR$, let 
$C[c,d]$ be the standard space of
(real) continuous functions on $[c,d]$ with the norm 
\begin{equation}
  \label{normcd}
\norm{f}:=\sup_{x\in  [c,d]}|f(x)|, 
\end{equation}
and let $\pn[c,d]$ denote $\pn$
regarded as a subspace of $C[c,d]$, \ie, equipped with the norm \eqref{normcd}.
The dual space of $C\cd$ is the space of signed measures on $\cd$, with the
total variation norm as in \eqref{neutrino}. Hence \eqref{neutrino} and
\eqref{abp} yield
\begin{equation}\label{psiab5}
\psi(a,b)
=\inf\bigset{\norm{\mu}_{C\oi^*}:\mu(p)=\chiab(p)\text{ for }p\in\pn\oi},
\end{equation}
which by the Hahn--Banach theorem 
yields
\begin{equation}\label{psiab2}
  \psi(a,b)=\norm{\chiab}_{\pn\oi^*}.
\end{equation}

We next identify $\chiab$. The definition \eqref{chiab} and the binomial
theorem yield, for $k=0,\dots,n$,
\begin{equation}\label{chiab3}
  \begin{split}
  \chiab\bigpar{x^{n-k}}
&=
  \chiab\bigpar{x^{n-k}(x+1-x)^k}
=\sum_{j=0}^k\binom kj\chiab\bigpar{x^{n-k+k-j}(1-x)^j}
\\&
=\sum_{j=0}^k \binom kj a^{n-k+k-j}b^j
=a^{n-k}(a+b)^k
\\&
=(a+b)^n \Bigparfrac{a}{a+b}^{n-k},
  \end{split}
\raisetag{\baselineskip}
\end{equation}
where the last equality assumes that $a+b\neq0$.
Consequently, if $a+b\neq0$, then
\begin{equation}\label{chiab4}
\chiab\bigpar{p}= (a+b)^n p\Bigparfrac{a}{a+b}
\end{equation}
for the monomials $p(x)=x^{n-k}$, and thus for all $p\in \pn$.
It can also be seen immediately that \eqref{chiab4} defines a linear
functional on $\pn$ that satisfies \eqref{chiab}.
Hence, in this case $\chiab$ is essentially a point evaluation at $a/(a+b)$,
and $\psi(a,b)$ is by \eqref{psiab2} given by the optimization problem
\begin{equation}\label{psiab6}
  \psi(a,b)=|a+b|^n\sup \Bigset{
	\Bigabs{p\Bigparfrac{a}{a+b}}:\max_{x\in\oi} |p(x)|=1},
\qquad a+b\neq0.
\end{equation}
Note that if $a,b\ge0$ (with $a+b>0$), then $a/(a+b)\in\oi$, 
so the supremum in \eqref{psiab5} is trivially 1, and thus
$\psi(a,b)=(a+b)^n$, as seen directly in \eqref{psi++}.
In contrast, in the case $a>0>b$, $a/(a+b)\notin\oi$, so \eqref{psiab6}
becomes an extrapolation problem.

In the case $a+b=0$, \eqref{chiab3} yields instead $\chiab(x^{n-k})=0$ for
$k\ge1$ and $\chiab(x^n)=a^n$. Hence, letting $[x^k]p(x)$ denote the
coefficient of $x^k$ in the polynomial $p(x)$,
\begin{equation}\label{chiaa}
  \chi_{a,-a} \bigpar{p(x)}=a^n [x^n]p(x).
\end{equation}
In other words, apart from a constant factor, $\chi_{a,-a}$ extracts the
coefficient of $x^n$. (This can also be seen as a limiting case of
\eqref{chiab4}, with $a/(a+b)\to\infty$.)

We consider the two cases separately, beginning with the case $b=-a$.
By homogeneity, it suffices to consider $a=1$.
By \eqref{psiab2} and \eqref{chiaa},
\begin{equation}
  \psi(1,-1)=\bignorm{p\mapsto[x^n]p(x)}_{\pn\oi^*}.
\end{equation}
The mapping $p(x)\mapsto p(2x-1)$ is an isometric bijection 
of $\pn[-1,1]$ onto $\pn\oi$. Since $[x^n]p(2x-1)=2^n [x^n]p(x)$,
it follows that we have
\begin{equation}\label{psi11}
  \psi(1,-1)=2^n\bignorm{p\mapsto[x^n]p(x)}_{\pn[-1,1]^*}.
\end{equation}
We thus want to find the largest possible coefficient of $x^n$ for a
polynomial of degree $n$ that is bounded by 1 on $[-1,1]$; equivalently, we
want to find the polynomial $p(x)$ with leading coefficient $x^n$ such that
$\norm{p}_{C[-1,1]}=\sup_{-1,1}|p(x)|$ is minimal. This is a classical
problem in approximation theory, 
which is solved by a multiple of the Chebyshev polynomial
$T_n(x):=\cos(n\arccos x)$, see \eg{} 
\cite[18.38(i)]{NIST} or \citet[Theorem 2.1]{Rivlin}.
Since $T_n$ has norm 1 in $\pn[-1,1]$ and its leading coefficient is $2^{n-1}$,
it follows that 
$p\mapsto[x^n]p(x)$ has norm $2^{n-1}$ on $\pn[-1,1]$, and thus
\eqref{psi11} yields
\begin{equation}\label{olo}
  \psi(1,-1)=2^{2n-1}.
\end{equation}
Consequently, \eqref{gamab} yields
\begin{equation}\label{ola}
\gam(n)\ge \frac{\psi\xpar{1,-1}}{2^n}=2^{n-1}.
\end{equation}

In order to see that equality holds in \eqref{ola}, we now consider the case
$a+b\neq0$, where we have shown \eqref{psiab6}. It suffices to consider the
case $|a|>|b|$ and $a>0>b$; then $\frac{a}{a+b}>1$.
We transfer again to $\pn\wii$ by the mapping $p(x)\mapsto p(2x-1)$ and see
that $\chi_{a,b}$ in \eqref{chiab4} then corresponds to
\begin{equation}\label{kk}
p\mapsto  
(a+b)^n p\Bigpar{2\frac{a}{a+b}-1}
=
(a+b)^n p\Bigpar{\frac{a-b}{a+b}}.
\end{equation}
Let $\xi:=\frac{a-b}{a+b}>1$. The problem is now to maximize $p(\xi)$ for
$p\in\pn$ with $\sup_{-1\le x\le 1}|p(x)|\le1$.
Again, the (unique)
extremal polynomial is the Chebyshev polynomial $T_n(x)$, see 
\cite[2.7.1]{Rivlin}; hence \eqref{psiab2} and \eqref{kk} yield
\begin{equation}\label{olb}
  \psi(a,b)=(a+b)^n T_n\Bigpar{\frac{a-b}{a+b}},
\qquad
a>0>b \text{ and }  a+b>0.
\end{equation}

Finally, we note that if $x>1$ and $y:=\arccosh x$, then $T_n(x)=T_n(\cosh
y)=\cosh(ny)$, and thus
\begin{equation}
  T_n(x)=\frac12\bigpar{e^{ny}+e^{-ny}}
\le \frac12\bigpar{e^{y}+e^{-y}}^n
=2^{n-1}x^n.
\end{equation}
Consequently, \eqref{olb} implies,
for $a>0>b$ and $a+b>0$,
\begin{equation}\label{olc}
  \psi(a,b)\le(a+b)^n 2^{n-1}\Bigpar{\frac{a-b}{a+b}}^n
=2^{n-1}(a-b)^n = 2^{n-1}(|a|+|b|)^n.
\end{equation}
It follows from \eqref{olc} and \eqref{olo} (which is a limiting case that also
follows from \eqref{olb} by continuity), together with the trivial case
$a,b\ge0$ treated earlier,   that 
$\psi(a,b)\le 2^{n-1}(|a|+|b|)^n$ for all real $a$ and $b$.
Consequently, \eqref{gamab} yields
\begin{equation}
  \gam(n)=\sup_{a,b\in\bbR}\frac{\psi(a,b)}{(|a|+|b|)^n}\le 2^{n-1}.
\end{equation}
By \eqref{olo} and \eqref{ola}, we have also the opposite inequality, and 
\refT{Tgamma} is proved.
\end{proof}

\begin{remark}\label{R466}
Since $T_n(x)=\frac12\bigpar{(x+\sqrt{x^2-1})^n+(x-\sqrt{x^2-1})^n}$,
the formula \eqref{olb} in  
the proof can be written (changing the sign of $b$)
  \begin{equation}\label{466}
	\begin{split}
	\normpisp{(a,-b)\nn}
&
=\psi(a,-b)
=(a-b)^n T_n\Bigpar{\frac{a+b}{a-b}}
\\&
=\frac{\bigpar{a+b+2\sqrt{ab}}^n+\bigpar{a+b-2\sqrt{ab}}^n}2
\\&
=\frac{\bigpar{\sqrt a+\sqrt b}^{2n}+\bigpar{\sqrt a-\sqrt b}^{2n}}2	  ,
	\end{split}
  \end{equation}
valid for any $a,b\ge0$ by symmetry, with the case $a=b$ following by
continuity or by \eqref{olo}.
\end{remark}

\begin{example}
  \label{Eab1}
We used in the proof of \refT{Tgamma}
the classical fact that $T_n(x)$ is extremal for
\eqref{psi11}. This can be seen as follows, which also yields an explicit
decomposition of the tensor product $(a,b)\nn$. 
(See \citet{Rivlin} for further details
and related results.)

We substitute $x=\cos\gth$; this yields an isometry $p\mapsto p(\cos\gth)$
of $\pn[-1,1]$ onto the space of trigonometric polynomials
\begin{equation}
\cT_n:=\biggset{\sum_{k=0}^n a_k \cos^k\gth:a_0,\dots,a_n\in\bbR}
=\biggset{\sum_{k=-n}^n b_{|k|} e^{\ii k\gth}:b_0,\dots,b_n\in\bbR}
\end{equation}
with the norm $\norm{q}_{\cT_n}=\sup_\gth|q(\gth)|$.
The linear functional $p\mapsto[x^n]p(x)$ on $\pn[-1,1]$ corresponds to the
linear functional $\chi$ 
mapping a trigonometric polynomial 
$q(\gth)=\sum_{k=0}^n a_k \cos^k\gth=\sum_{k=-n}^n b_{|k|} e^{\ii k\gth}$
to $a_n=2^nb_n$.
A simple calculation (a Fourier inversion in $\bbZ_{2n}$) yields
\begin{equation}\label{chib}
  \begin{split}
	\frac{1}{2n}\sum_{j=0}^{2n-1} (-1)^jq\Bigparfrac{j\pi}{n} 
=
\frac{1}{2n}\sum_{j=0}^{2n-1}\sum_{k=-n}^n b_{|k|} e^{\ii j(k+n)\pi/n}
=2b_n 
  \end{split}
\end{equation}
and thus 
\begin{equation}
  |b_n| \le \frac12\norm{q},
\end{equation}
with equality for $q(\gth)=\cos(n\gth)$.
Consequently, the linear functional $q\mapsto b_n$ has norm $\frac12$ on
$\cT_n$, so the linear functional $q\mapsto a_n=2^nb_n$ has norm $2^{n-1}$.
As said above, this corresponds by an isometry  to the linear functional
$[x^n]p(x)$ on $\pn[-1,1]$, so this functional too has norm $2^{n-1}$ and
\eqref{olo} follows.

We see also from \eqref{chib} that for any $p\in\pn$, with
$q(\gth)=p(\cos\gth)$, 
\begin{equation}\label{chia}
  \begin{split}
[x^n]p(x)=2^n b
=\frac{2^{n-1}}{2n}\sum_{j=0}^{2n-1} (-1)^jq\Bigparfrac{j\pi}{n} 
=\frac{2^{n-1}}{2n}\sum_{j=0}^{2n-1} (-1)^jp\Bigpar{\cos\frac{j\pi}{n}} 
.  \end{split}
\end{equation}
Transforming back to \oi, this yields
\begin{equation}\label{chic}
  \begin{split}
\chi_{1,-1}(p)
&=
[x^n]p(x)
=\frac{2^{2n-1}}{2n}\sum_{j=0}^{2n-1} (-1)^j
p\biggpar{\frac{1+\cos\frac{j\pi}{n}}2} 
\\&
=\frac{2^{2n-1}}{2n}\sum_{j=0}^{2n-1} (-1)^j
p\Bigpar{\cos^2\frac{j\pi}{2n}} 
.  \end{split}
\end{equation}
This yields an optimal representation of $\chi_{1,-1}$ as a signed measure
$\mu$ on \oi, which by the argument above corresponds to an optimal
decomposition of $(1,-1)\nn$ into positive tensor powers:
\begin{equation}\label{chit}
  (1,-1)\nn
=
\frac{2^{2n-1}}{2n}\sum_{j=0}^{2n-1} (-1)^j
\Bigpar{\cos^2\frac{j\pi}{2n},\sin^2\frac{j\pi}{2n}}\nn .
\end{equation}
(Note that there are only $n+1$ different tensor powers on the \rhs, since
the terms for $j$ and $2n-j$ are equal in \eqref{chit}, as well as in
\eqref{chia} and \eqref{chic}.) Moreover, it follows also from this argument
that this  optimal decomposition is unique.
\end{example}

\begin{example}
  We can similarly find an optimal decomposition of $(a,-b)\nn$ for arbitrary
$a,b>0$. Assume $a-b\neq0$; then \eqref{kk} (with $-b$ instead of $b$) and
the arguments above
show that we want to represent
the linear functional $p\mapsto p(\xi)$ 
on $\pn\wii$
for a given $\xi=\frac{a+b}{a-b}$ 
with $|\xi|>1$.
Again we seek a representation as a linear combination of
$p\bigpar{\cos \frac{j\pi}n}$, $j=0,\dots,n$, since these are the points
where $|T_n(x)|$ attains its maximum on $\wii$, Thus, again extending the
summation to $j=0,\dots,2n-1$ for convenience, 
we want to find $c_j(\xi)$, with $c_{2n-j}(\xi)=c_j(\xi)$, 
such that
\begin{equation}\label{pyrt}
  p(\xi)
=\sum_{j=0}^{2n-1} c_j(\xi)p\Bigpar{\cos\frac{j\pi}{n}},
\qquad p\in\pn.
\end{equation}
In fact, if \eqref{pyrt} holds, then
it extends to vector-valued polynomials (by considering each component
separately);
taking $p$ to be the vector-valued polynomial
$\bigpar{\frac{1+x}2,\frac{1-x}2}\nn$ then yields
\begin{equation}\label{chiq}
  (a,-b)\nn
=
\sum_{j=0}^{2n-1} (a-b)^n c_j\Bigpar{\frac{a+b}{a-b}}
\Bigpar{\cos^2\frac{j\pi}{2n},\sin^2\frac{j\pi}{2n}}\nn .
\end{equation}
Since $\pn$ has dimension $n+1$, there exists a unique such representation
\eqref{pyrt}. 
Moreover, the general theory, see \cite[Chapter 2]{Rivlin} for details, 
or alternatively the calculations at the end of this example,
shows that the representation \eqref{pyrt} is optimal in the sense that
$\sum_j|c_j(\xi)|$ equals the norm of $p\mapsto p(\xi)$ on $\pn\wii$;
furthermore, this is the unique optimal representation.
Consequently, \eqref{chiq} yields the unique optimal decomposition of
$(a,-b)\nn$.

In order to find $c_j(\xi)$, we take $p(x)=T_k(x)=\cos(k\arccos x)$ in
\eqref{pyrt} and find
\begin{equation}\label{tcos}
  T_k(\xi) 
=\sum_{j=0}^{2n-1} c_j(\xi)\cos\frac{jk\pi}{n},
\qquad k=0,\dots n.
\end{equation}
Furthermore, by our choice $c_{2n-j}(\xi)=c_j(\xi)$, 
$
\sum_{j=0}^{2n-1} c_j(\xi)\sin\frac{jk\pi}{n}=0
$
for any $k$; hence \eqref{tcos} yields
\begin{equation}
\sum_{j=0}^{2n-1} c_j(\xi)e^{-\ii jk\pi/n}
=T_{|k|}(\xi),
\qquad k=-n,\dots n.
\end{equation}
A Fourier inversion (on $\bbZ_{2n}$) now yields
\begin{equation}\label{cta}
c_j(\xi)=
\frac{1}{2n}
\sum_{k=-n}^{n-1} e^{\ii jk\pi/n}T_{|k|}(\xi)
.
\end{equation}
Substituting this in \eqref{chiq} yields the optimal decomposition of
$(a,-b)\nn$ for any $a,b>0$, with the case $a=b$ in \eqref{chit}
interpreted as a limit.

We can calculate the coefficients $c_j(\xi)$ in \eqref{cta} more explicitly.
Suppose that $a>b>0$, so $\xi=\frac{a+b}{a-b}>1$, and let 
$y:=\arccosh \frac{a+b}{a-b}>0$.
Then $T_k(\xi)=\cosh(ky)$, and thus \eqref{cta} yields
\begin{align}\label{chiqq}
c_j\Bigpar{\frac{a+b}{a-b}}&
=
\frac{1}{2n}
\sum_{k=-n}^{n-1} e^{\ii jk\pi/n}\cosh (ky)
=
\frac{1}{4n}
\sum_{k=-n}^{n-1} \bigpar{e^{\ii jk\pi/n+ky}+e^{\ii jk\pi/n-ky}}
\notag\\&
=\frac{1}{4n}\frac{e^{\ii j\pi}\bigpar{e^{ny}-e^{-ny}}}{e^{\ii j\pi/n+y}-1}
+ \frac{1}{4n}\frac{e^{\ii j\pi}\bigpar{e^{-ny}-e^{ny}}}{e^{\ii j\pi/n-y}-1}
\notag\\&
=(-1)^j\frac{\sinh(ny)}{2n}
\Bigpar{\frac{1}{e^{\ii j\pi/n+y}-1}-\frac{1}{e^{\ii j\pi/n-y}-1}}
\notag\\&
=(-1)^j\frac{\sinh(ny)}{2n} \frac{e^{\ii j\pi/n-y}-e^{\ii j\pi/n+y}}
 {e^{2\ii j\pi/n}+1-e^{\ii j\pi/n}\bigpar{e^y+e^{-y}}}
\notag\\&
=(-1)^j\frac{\sinh(ny)}{2n} \frac{\sinh y}
{\cosh y-\cos(j\pi/n)}.
\end{align}
In particular, note that $\sign\bigpar{c_j(n)}=(-1)^j$, so
$c_j(\xi)$ alternates in sign. This shows by
\eqref{pyrt} and the fact that $T_n(\cos(j\pi/n))=(-1)^j$,
\begin{align}
\bignorm{p\mapsto p(\xi)}_{\pn[-1,1]^*}
=\sum_{j=0}^{2n-1}\bigabs{c_j(n)}
\end{align}
and thus, by \eqref{psiab2} and \eqref{kk} (still with $b$ replaced by $-b$)
\begin{align}\label{chiqz}
\psi(a,-b)
=(a-b)^n\sum_{j=0}^{2n-1}\bigabs{c_j(n)}.
\end{align}
This
verifies directly that the decomposition \eqref{chiq} is optimal,
without the general theory referred to above.
\end{example}

\begin{remark}
  Another expression for $c_j(\xi)$ can be obtained using the Lagrange
  interpolation polynomials $\ell_k(x)$ for the points
  $x_j=\cos\frac{j\pi}n$, $j=0,\dots,n$,
 see \cite[\S3.3]{NIST}; 
these are 
given by
$\ell_k(x)=\prod_{j\neq k}\frac{x-x_j}{x_k-x_j}$ and are
  characterized 
as the polynomials in $\pn$
satisfying $\ell_k(x_j)=\gd_{jk}$, and thus, for any polynomial $p\in\pn$
and any real (or complex) $\xi$,
  \begin{equation}
	p(\xi)=\sum_{j=0}^n \ell_j(\xi)p(x_j).
  \end{equation}
Consequently, $c_j(\xi)=\ell_j(\xi)$, now summing for $j=0,\dots,n$ only.
\end{remark}

\section{Exchangeable random variables}\label{Sexch}

\subsection{More notation}\label{SSnot2}
Let $S=(S,\cS)$ be an arbitrary measurable space. 
$\cM(S)$ denotes the Banach space of
(finite) signed measures on $S$, 
with $\norm{\mu}$ defined to be the total variation
of $\mu$. Furthermore, $\cP(S)$ is the subset of probability measures on
$S$, \ie, the positive measures with norm 1. 
We regard $\cM(S)$ and $\cP(S)$ as measurable spaces with the
$\gs$-fields generated by the evaluations $\mu\mapsto \mu(A)$ for measurable
$A\subseteq S$ (\ie, $A\in\cS$). 
Recall that if $X$ is a random element of $S$, then its distribution is a
measure in $\cP(S)$.

If $x\in S$, then $\gd_x$ denotes the Dirac measure, \ie, unit point mass,
at $x$. (This is the distribution of the non-random $X:=x$.)

For a finite (or countable) set $S$, we identify the space $\cM(S)$ of signed
measures on $S$ with $\ell_1(S)$.
In particular, $\gd_x$ is  identified with the vector 
$(\indicq{y=x})_{y\in S}\in\ell_1(S)$, and thus
$\gd_i=e_i$ when $S=\bbN$.

Let $\setnn:=\setn$.

\subsection{Finitely exchangeable distributions}

Let $S=(S,\cS)$ be a measurable space. A random vector $\XX=(X_1,\dots,X_n)$
with values in $S^n$ is \emph{(finitely) exchangeable} if its distribution
is symmetric under permutations. 
See \eg{} \citet{Aldous} for a survey of both finite and infinite
exchangeability.

For an infinite exchangeable sequence $\XX=(X_i)_1^\infty$, the well-known de
Finetti's theorem says that under weak technical conditions on $S$ (for
example that $S$ is a Borel space), the distribution is a mixture of 
product (power) measures, see \eg{} 
\cite[\S2]{Aldous} or
\cite[Theorem 1.1]{Kallenberg-symmetries}.
In formulas, this says that 
if $\cP(S)$ is the space of probability measures on $S$, and 
$\muXX\in\cP(S^\infty)$ is the distribution of $\XX$, then
there exists a probability measure $\gl$ on $\cP(S)$
such that
\begin{align}
  \label{deFinetti}
\muXX=\int_{\cP(S)} \nu^\infty\dd\gl(\nu) .
\end{align}
It is also well-known that this, in general, fails for finitely exchangeable
sequences, see \eg{} \cite{Diaconis,DiaconisF}. 
A substitute in the finite case is that there always exists such
a representation with a \emph{signed} measure $\gl$.
To be precise, see 
\cite[V.52]{DM},
\cite{Jaynes}, \cite{KS06}, \cite{SJ308},
if $\XX=(X_1,\dots,X_n)$ is exchangeable, with values in an arbitrary
measurable space $S$, then 
there exists a signed measure $\gl$ on $\cP(S)$,
\ie, $\gl\in\cMPS$,
such that
\begin{align}
  \label{deFinetti+-}
\muXX=\int_{\cP(S)} \nu^n\dd\gl(\nu) .
\end{align}
A natural question (posed in \cite{SJ308})
is how large the total variation $\norm\gl$ of $\gl$ 
has to be.
Since $\muXX$ is a probability measure, we always have 
$\int\dd\gl(\nu)=1$, and thus $\norm{\gl}\ge1$, with equality if and only if
$\gl$ is a probability measure (as in de Finetti's theorem
\eqref{deFinetti}).
Hence, $\norm\gl$ is a measure of how far the representation 
is from the ideal representation as a mixture of powers.
Note that $\gl$ is not unique, so we are interested in the optimal $\gl$, or
more generally $\inf\norm\gl$ over all possible representing $\gl$ in
\eqref{deFinetti+-}.

An answer to this question is given by the following theorem, which connects
this problem to the  tensor norms studied above.

\begin{theorem}\label{Tex}
  \begin{thmenumerate}
  \item\label{Texa}%
If\/ $\XX=(X_1,\dots,X_n)$ is exchangeable, with values in an arbitrary
measurable space $S$, then 
its distribution $\muXX\in\cP(S^n)$
has a representation \eqref{deFinetti+-} with a 
signed measure $\gl$ on $\cP(S)$
such that
\begin{align}\label{tex1}
  \norm{\gl}_{\cM(\cP(S))}\le \kkn,
\end{align}
where, as in \eqref{tcsp1},
\begin{align}\label{tex2}
 \kkn:=
\csp(n,\ell_1)
=\normpispq{e_1\vees e_n}{\ell_1^n}.
\end{align}
The constant $\kkn$ given in \eqref{tex2} is, in general, the best
possible. We have
\begin{align}
  \label{tex3}
\frac{n^n}{n!}\le\kkn\le 2^{n-1}\frac{n^n}{n!}.
\end{align}
\item\label{Texm}%
If furthermore $S$ is finite, with $|S|=m$, then \eqref{tex1} can be
replaced by
    \begin{align}\label{texm}
  \norm{\gl}_{\cM(\cP(S))}\le \csp(n,\ell_1^m).
\end{align}
Moreover, this constant is the best possible for the given $S$.
If\/ $m\ge n$, then this constant equals $\kk(n)$.
  \end{thmenumerate}
\end{theorem}

By \refE{Ekk2},  $\kk(2)=3$; hence neither of the bounds in \eqref{tex3} is
sharp.

\begin{problem}
  What is the exact value of $\kkn$?
\end{problem}

It follows from \eqref{tex3} and Stirling's formula that,
recalling \eqref{cspoo},
  \begin{align}\label{kang}
e\le \limsup_\ntoo \kkn^{1/n} =\csp(\ell_1)\le 2e.
  \end{align}
\begin{problem}
What is $\limsup_\ntoo \kkn^{1/n}$?
Does $\lim_\ntoo \kkn^{1/n}$ exist?
\end{problem}

Before proving \refT{Tex}, consider first for simplicity  
the case when $S$ is finite. 
Then, a distribution (\ie, probability measure) 
$\mu$ on $S^n$ is the same as a positive element of
norm 1 in $\ell_1(S^n)$.
Since $S$ is finite, $\ell_1(S^n)=\ell_1(S)\nnpi$, isometrically.
Thus, 
a distribution
$\mu$ on $S^n$ is the same as a positive element of
norm 1 in $\ell_1(S)\nnpi$.
Furthermore, by definition, $\mu$ is exchangeable if it is invariant under 
permutations of the coordinates, which is the same as saying that $\mu$,
regarded as a tensor in $\ell_1(S)\nn$, is a symmetric tensor.
Hence, an exchangeable distribution $\mu$ is a positive element of
$\ell_1(S)\snn$ with $\normpi{\mu}=1$.

Consider now representations as in \eqref{deFinetti+-} of an exchangeable 
distribution $\muXX$.
If $\gl$ has finite support, then \eqref{deFinetti+-} becomes a
representation as in \eqref{normpisp}, and thus 
$\normpispq{\muXX}{\ell_1(s)}\le\norm\gl$.
Furthermore, this extends to arbitrary measures $\gl$ since \eqref{deFinetti+-}
implies 
\begin{align}\label{hedemora}
\normpisp{\muXX}\le\int \normpisp{\nu^n}\dd|\gl|(\nu)=\norm{\gl}.
\end{align}
(The spaces are finite-dimensional and there are no problems with
measurablilty or convergence.)
Conversely, a representation as in \eqref{normpisp} yields a representation
\eqref{deFinetti+-} with $\gl=\sum_k a_k\norm{x_k}^n\gd_{x_k/\norm{x_k}}$
and thus $\norm{\gl}\le\sum_k |a_k|\norm{x_k}^n$.
Consequently,
when $S$ is finite,
\begin{align}\label{puh}
  \inf\bigset{\norm{\gl}: \text{\eqref{deFinetti+-} holds}}
=\normpispq{\muXX}{\ell_1(S)}.
\end{align}
Moreover, \refR{R=+}
implies that the infimum in \eqref{puh} is attained by some $\gl$; 
in fact, by some $\gl$ with finite support.

We have shown that if $S$ is finite, then \eqref{deFinetti+-}
holds with $\norm{\gl}=\normpisp{\muXX}$.
A special case
is to take $S=\setnn:=\setn$
and let the random vector $(X_1,\dots,X_n)$ be a uniformly random permutation
of $\setn$, which means that $\muXX:=e_1\vees e_n$.
This case is easily seen to be extreme. In fact, if $S$ is any finite set
and $\xx=(x_1,\dots,x_n)\in S^n$, then 
\begin{align}
  \label{gfxx}
\gfxx (e_i):=\gd_{x_i}
\end{align} 
defines a
linear operator 
$\gf_{\xx}:\ell_1^n\to\cM(S)=\ell_1(S)$ with $\norm{\gf_{\xx}}=1$, 
and thus, by \refT{Tfunctor+},
$ 
\normpisp{\gd_{x_1}\vees \gd_{x_n}} \le \normpisp{e_1\vees e_n}.
$
Furthermore, every exchangeable distribution $\muXX$ on $S^n$ is
a convex combination of tensors of the type $\gd_{x_1}\vees \gd_{x_n}$. 
Consequently,
\begin{align}
\normpisp{\muXX} \le \normpisp{e_1\vees e_n}.
\end{align}
This proves,
together with \eqref{puh},
the main assertion in \refT{Tex} when $S$ is finite.

The general proof uses the same idea; we only have to add some
technicalities, which we borrow from \cite{SJ308}, where further
details may be found if necessary; see also \cite{KS06}.

\begin{proof}[Proof of \refT{Tex}]
\pfitemref{Texa}
  Fix a representation
  \begin{align}\label{fixed}
    e_1\vees e_n 
=\sumkN a_k \ff_{k}\nn,
  \end{align}
where $a_k\in\bbR$ and $\ff_k\ge0$ are unit vectors in
$\ell_1^n=\cM(\setnn)$;
thus $\eta_k\in\cP(\setnn)$.

For any $\xx=(x_1,\dots,x_n)\in S^n$, define again the linear map
$\gfxx:\cM(\setnn)\to\cM(S)$ by \eqref{gfxx} and linearity, and note that
$\gfxx$ maps $\cP(\setnn)$ into $\cP(S)$. ($\gfxx$ is the natural
push-forward of measures induced by the mapping $\setnn\to S$ given by
$i\mapsto x_i$.)
Furthermore,
$\gfxx\nn:\cM(\setnn)\nn\to \cM(S)\nn$ 
and we may regard $\cM(S)\nn$ as a subspace of $\cM(S^n)$ also when $S$ is
infinite. 

Define further, using the decomposition \eqref{fixed},
\begin{align}\label{psixx}
  \psi_{\xx}:=\sumkN a_k \gd_{\gfxx(\ff_k)}\in\cM(\cP(S)).
\end{align}
Then, for any $\xx\in S^n$, using \eqref{fixed} and \eqref{gfxx},
\begin{align}
  \int_{\cP(S)} \nu^n \dd\psi_\xx(\nu)
&=\sumkN a_k\gfxx(\ff_k)\nn
=\sumkN a_k\gfxx\nn\bigpar{\ff_k\nn}
\notag\\&
=\gfxx\nn\bigpar{e_1\vees e_n}
=\gfxx(e_1)\vees\gfxx(e_n)
\notag\\&
=\gd_{x_1}\vees\gd_{x_n}.
\label{aslog}
\end{align}
Furthermore, 
for each fixed $\ff\in\cP(S)$, the map $\xx\mapsto\gfxx(\ff)$
is measurable $S^n\to \cP(S)$, and thus
the map $\xx\mapsto\psi_\xx$ is measurable $S^n\to\cM(\cP(S))$.
Hence, $\psi_{\XX}$ is a random measure in $\cM(\cP(S))$.
Moreover, by \eqref{psixx},
\begin{align}\label{kraka}
\norm{\psi_{\XX}}_{\cM(\cP(S))}
\le K:= \sumkN |a_k|.
\end{align}
Hence, we can define the expectation $\gl:=\E\psi_{\XX}\in \cM(\cP(S))$,
\cf{} \cite[Lemma 2.4]{Kallenberg:RM}.
Furthermore, \eqref{kraka} implies
$\norm{\gl}_{\cM(\cP(S))}\le K$, and \eqref{aslog} implies
\begin{align}\label{disa}
  \int_{\cP(S)} \nu^n \dd\gl(\nu)&
=\E  \int_{\cP(S)} \nu^n \dd\psi_{\XX}(\nu)
=\E\bigpar{ \gd_{X_1}\vees\gd_{X_n}}
=\muXX.
\end{align}
This shows the existence of a representation \eqref{deFinetti+-} with
$\norm{\gl}\le K$, given by \eqref{kraka}.

We may, by \refR{R=+}, choose the decomposition \eqref{fixed} such that
$K=\normpisp{e_1\vees e_n}=\kkn$, and thus \eqref{tex1} holds.

To see that $\kkn$ is best possible, it suffices to take $S=\setnn$ and
$\muXX=e_1\vees e_n$, as in the discussion before the proof. Then
\eqref{puh} shows that every representating measure $\gl$ satisfies
$\norm\gl\ge\kkn$.

Finally, \eqref{tex3} follows from \eqref{abc}, \refT{Tgamma} 
and \eqref{gabriel}.

\pfitemref{Texm}
By \eqref{puh} and the comment after it, we can find $\gl$ with
\begin{align}\label{ior}
\norm{\gl}\le\normpispq{\muXX} {\ell_1(S)}
\le  \csp\bigpar{n,\ell_1(S)}
=  \csp\bigpar{n,\ell_1^m}.
\end{align}

On the other hand, if $M$ is a constant such that there always exists a
$\gl$ with
$\norm\gl\le M$, then \eqref{puh} shows that 
$\normpisp{\mu}\le M$
for every positive
$\mu\in\ell_1(S)\snn$ with $\normpi\mu=1$.
This extends to all 
$\mu\in\ell_1(S)\snn$ with $\normpi\mu=1$, by decomposing them in their
positive and negative parts, and thus
$\csp(n,\ell_1^m)=\csp\xpar{n,\ell_1(S)}\le M$.

Finally, 
if $m\ge n$ then
$\csp(n,\ell_1^m)=\kk(n)$ by \refT{Tcsp}.
\end{proof}

\begin{remark}
   The proof in \cite{SJ308} of the representation \eqref{deFinetti+-} used
  the argument above, with a decomposition \eqref{fixed} where $\ff_k$
  ranged over the $\binom{2n-1}{n-1}$ probability measures $\nu$
in $\cP(\setnn)$ such that $n\mu$ is integer-valued; it was shown in
\cite{SJ308} by an algebraic argument that there always exists a unique such
decomposition. No attempt was made in \cite{SJ308} to evaluate the best
constant; in fact, a numerical calculation (using Maple) of
the constant $K=K_n$ in \eqref{kraka} for the decomposition in \cite{SJ308} yields \eg{}
$K_2=3$, $K_3=20$, $K_4=210$, $K_5=3024$. 
These values are thus upper bounds for $\kkn$; we see that for $n=2$, we
obtain the sharp constant $\kkx2=3$ (see \refE{Ekk2}),
but already for $n=3$, this $K_n$ is larger than the upper bound in
\eqref{tex3} ($\kkx3\le18$). In other words (not surprisingly), the
decomposition used in 
\cite{SJ308} is not optimal.
\end{remark}

\begin{remark}
  Note that the proof uses the $\gs$-field on $\cMPS$ defined in
  \refSS{SSnot2}, and not the (in general larger) 
Borel $\gs$-field on the Banach space $\cMPS$; 
in general, the mapping $\xx\to\psi_\xx$ is not measurable if $\cMPS$ is
given the latter $\gs$-field.
\end{remark}

\begin{remark}\label{R+-+-}
  We have considered representations \eqref{deFinetti+-} where $\gl$ is a
  signed measure but $\nu$ ranges over probability measures.
An alternative is to allow also $\nu$ to be a signed measure, \ie, to
consider representations
\begin{align}
  \label{deFinetti+-+-}
\muXX=\int_{B(\cM(S))} \nu^n\gl(\ddx \nu) 
\end{align}
where $B(\cM(S))$ denotes the unit ball in the Banach space $M(S)$ of signed
neasures on $S$.
The arguments above are easily modified to this case and show that there
always exists such a representation with
\begin{align}  \label{tex+-}
  \norm{\gl}_{\cM(B(\cM(S)))}
\le 
\normpisq{e_1\vees e_n}{\ell_1}
=\cs(n,\ell_1)=\frac{n^n}{n!},
\end{align}
where we used \refE{El1} for the explicit value; moreover, this constant is
the best possible.
In particular, this shows that if $n\ge2$, then we cannot in general find a
representation \eqref{deFinetti+-+-} where $\gl$ is a probability measure on
$B(\cM(S))$. 
\end{remark}

\begin{remark}\label{Rbetter}
  The upper bound in \eqref{tex3} can be improved a little as follows.

By \eqref{polt},
\begin{align}\label{uggla}
\kkn=
\normpisp{e_1\vees e_n}
\le
\frac{1}{2^nn!}
\sum_{\eps_1,\dots,\eps_n=\pm1} 
\Bignormpisp{\biggpar{\sumin \eps_ie_i}\nn}.
\end{align}
Consider one of the terms in the sum, and suppose that $\eps_i=1$ for $k$
indices $i$.
The argument in the beginning of the proof of \refT{Tcssp}, up to the first
inequality in \eqref{kroken}, with $u:=(k,-(n-k))\in\ell_1^2$, 
show that, using \eqref{psi}, 
\begin{align}
  \Bignormpispq{\biggpar{\sumin \eps_ie_i}\nn}{\ell_1^n}
\le \normpispq{u\nn}{\ell_1^2}=\psi\bigpar{k,-(n-k}).
\end{align}
This is evaluated in \eqref{466}, 
and thus \eqref{uggla} yields, by counting terms,
\begin{align}
\kkn&
\le
\frac{1}{2^nn!} \sumkon \binom{n}k \psi\bigpar{k,-(n-k)}
\notag\\&
=
\frac{1}{2^{n+1}n!} \sumkon \binom{n}k 
\Bigpar{\Bigpar{\sqrt{k}+\sqrt{n-k}}^{2n}+\Bigpar{\sqrt{k}-\sqrt{n-k}}^{2n}}
. \label{uv} 
\end{align}
For $n=2$, \eqref{uv} yields the correct value 3. 
We have no reason to
believe that the bound is sharp for larger $n$.

The improvement from the upper bound in \eqref{tex3} lies in that we here
use the exact value \eqref{466} for each term, while the proof of
\eqref{tex3} estimates each $\psi(k,-(n-k))$ by the worst case $k=n/2$.
However, the improvement is slight, since most terms in \eqref{uggla} have 
$k$ close to $n/2$. In fact, simple asymptotic estimates (which we omit)
show that asymptotically, \eqref{uv} improves the upper bound only by a
factor $\sqrt{2/3}\doteq0.816$. Numerically, the improvement factor is close
to this value also for small $n$, with a factor $0.75$ for $n=2$ and $3$.
\end{remark}

\subsection{Binary variables}\label{SSbinary}
If $S$ is finite with $|S|=m<n$,
we may hope that the bound $\csp(n,\ell_1^m)$ in \eqref{texm}
is better than $\kk(n)$.
We consider here only the simplest case $|S|=2$, for example $S=\set{0,1}$.

Let, for $0\le j\le n$, $\mu_j$ be the distribution of a random vector
$\XX\in S^n$ consisting of $j$ 0's and $n-j$ 1's in random order; 
thus, 
\begin{align}\label{mum}
\mu_j=\gL\bigpar{\gd_0\nnx{j}\tensor\gd_1\nnx{n-j}}.
\end{align}
Evidently,
$\mu_j$ is exchangeable. Moreover, every exchangeable distribution on $S^n$
is a mixture of these measures $\mu_j$, and it follows that
\begin{align}\label{snork}
\csp(n,\ell_1^2)=
\csp(n,\ell_1(S))=
\sup_{0\le j\le n}  \normpispq{\mu_j}{\ell_1(S)}.
\end{align}
We thus want to find $\normpispq{\mu_j}{\ell_1(S)}$.

We argue as in the proof of \refT{Tgamma}. This yields, \cf{}
\eqref{neutrino}--\eqref{ab}, that
\begin{align}\label{mumin}
  \normpispq{\mu_j}{\ell_1(S)}
=\inf\norm{\mu}=\inf\intoi|\ddx\mu|(x),
\end{align}
taking the infimum over all signed measures $\mu$ on $\oi$ such that
\begin{align}
  \intoi(x,1-x)\nn\dd\mu(x)=\mu_j,
\end{align}
which is equivalent to, by expanding into coordinates in $(\bbR^2)\nn$,
\begin{align}\label{apa}
  \intoi x^k(1-x)^{n-k}\dd\mu(x)=\frac{1}{\binom{n}{j}}\gd_{kj},
\qquad k=0,\dots,n.
\end{align}

Let again $T_n$ be the Chebyshev polynomial.
Let $x\in\oi$, write $y:=1-x$, $t:=\arccos(2x-1)$ and $s:=t/2$.
Then, 
$\cos^2s=(1+\cos t)/2=x$, $\sin^2s=1-\cos^2s=y$,
and thus
$e^{\ii s}=x\qq+\ii y\qq$. 
Consequently,
\begin{align}
  T_n(2x-1)&
=\cos(nt) 
=\cos(2ns) 
=\Re e^{\ii 2ns}
=\Re\bigpar{x\qq+\ii y\qq}^{2n}
\notag\\&
=\sum_{k=0}^{n/2} \binom{2n}{2k}x^k(-y)^{n-k}
\notag\\&
=\sum_{k=0}^{n/2} \binom{2n}{2k}(-1)^{n-k}x^k(1-x)^{n-k}.
\end{align}
Hence, if $\mu$ satisfies \eqref{apa}, then
\begin{align}
  \intoi T_n(2x-1)\dd\mu(x) = 
\frac{(-1)^{n-j}}{\binom{n}{j}} 
\binom{2n}{2j}
\end{align}
which implies, since $|T_n(2x-1)|\le1$ for $x\in\oi$,
\begin{align}\label{bnm}
  \intoi |\ddx\mu|(x) \ge
\frac{\binom{2n}{2j}}{\binom{n}{j}}.
\end{align}
Recalling \eqref{mumin}, we have shown that
\begin{align}\label{mumrik}
\normpispq{\mu_j}{\ell_1(S)}
\ge
 \frac{\binom{2n}{2j}}{\binom{n}{j}}.
\end{align}
Thus, by \eqref{snork},
using an elementary calculation to optimize $j$,
\begin{align}\label{mumrik1}
\csp\bigpar{n,\ell_1^2}
\ge
\max_{0\le j\le n}
 \frac{\binom{2n}{2j}}{\binom{n}{j}}
= \frac{\binom{2n}{2\floor{n/2}}}{\binom{n}{\floor{n/2}}}.
\end{align}

We conjecture that $T_n(2x-1)$ is extremal here too, so that equality holds
in \eqref{mumrik} and \eqref{mumrik1}, but we leave that as an open problem.

In any case, \eqref{mumrik1} is a lower bound. Stirling's formula yields
the asymptotic estimate
\begin{align}\label{mumrik3}
\csp\bigpar{n,\ell_1^2}
\ge 2^{n-1/2+o(1)}.
\end{align}
Hence, the constants $\csp\bigpar{n,\ell_1^2}$ also grow exponentially, but
possibly (presumably) at a slower rate than
$\kkn=\csp\bigpar{n,\ell_1}$, see \eqref{kang}.

However, a numerical calculation reveals that for $2\le n\le 4$, the lower
bound $n^n/n!$ in \eqref{tex3} is smaller than the bound in \eqref{mumrik1}.
We thus have, using also \refE{Ekk2} or \refR{Rbetter}
for $n=2$, the improved bounds
\begin{align}
  \kk(2)&=\csp\bigpar{2,\ell_1^2}=3,
\\
  \kk(3)&\ge\csp\bigpar{3,\ell_1^2}\ge5,
\\
  \kk(4)&\ge\csp\bigpar{4,\ell_1^2}\ge\frac{35}3.
\end{align}

\begin{problem}
Find a non-trivial upper bound for $\csp(n,\ell_1^2)$.
  Is, as conjectured above, \eqref{mumrik1} an equality?
\end{problem}

\begin{problem}
  Extend this to $\csp(n,\ell_1^m)$ for other fixed values of $m$.
\end{problem}

\subsection{Extendible finitely exchangeable variables}\label{SSextend}

Let $n$ and $N$ be positive integers with $N\ge n$.
An exchangeable random vector $\XX_n=(X_1,\dots,X_n)$ in $S^n$ is
\emph{$N$-extendible} if it can be extended to an exchangeable random vector
$\XX_N=(X_1,\dots,X_N)$. 
We similarly say that an exchangeable distribution on $S^n$ is
$N$-\ext{}
if it is the distribution of an $N$-\ext{} vector.
Note that by de Finetti's theorem \eqref{deFinetti}, 
at least if $S$ is a Borel space, a
distribution is $\infty$-\ext{} if and only if it has a representation
\eqref{deFinetti+-} with a probability measure $\gl$. 
However, we will here consider the case of finite $N$.
See \eg{} \cite{Diaconis,DiaconisF,Takis2} for various aspects of
extendibility.

Let $\cE_n=\cE_n(S)$ 
be the set of \exch{} distributions on $S^n$, and let $\cE_{n,N}=\cE_{n,N}(S)$
be the subset of \Next{} distributions.
Let $\PiNn:\cP(S^N)\to\cP(S^n)$ be the map induced by projecting a random
vector $(X_1,\dots,X_N)$ onto its first $n$ coordinates.
Thus $\cE_{n,N}=\PiNn(\cE_N)\subseteq\cE_n$.

Consider again first the case when $S$ is finite. 
Then, as discussed above, $\cE_n$ is the set of positive unit elements in
$\ell_1(S)\nn$. 

Consider the special case $S=[N]$, and define
\begin{align}
  \chi_{n,N}:=\PiNn\bigpar{e_1\vees e_N} \in \cE_{n,N}([N]).
\end{align}
This is thus the distribution of $(X_1,\dots,X_n)$ when $(X_1,\dots,X_N)$ is
a uniformly random permutation of $[N]$; in other words, $\chi_{n,N}$ is the
distribution of the random vector obtain by drawing $n$ elements of $S=[N]$
without replacement.
We will see that this is, not surprisingly, an extreme case, 
\cf{} \eg{} \cite{DiaconisF}.
Let
\begin{align}\label{kknN}
  \kknN:=\normpispq{\chi_{n,N}}{\ell_1^N}.
\end{align}

For an arbitrary $S$ and $\xx=(x_1,\dots,x_N)\in S^N$, define 
$\gfxx:\cM(\setNN)\to\cM(S)$ by \eqref{gfxx}
and linearity.
Then, 
$\gfxx\nn(\chi_{n,N})\in\cM(S)\nn\subseteq\cM(S^n)$
is the distribution of the random vector obtained by drawing $n$ elements
of $x_1,\dots,x_N$ without replacement,
see \eqref{kik} below.

\begin{theorem}\label{Texex}
  \begin{thmenumerate}
  \item \label{Texexa}
Let $1\le n\le N$.
If\/ $\XX=(X_1,\dots,X_n)$ is exchangeable and \Next, 
with values in an arbitrary measurable space $S$, then 
its distribution $\muXX\in\cP(S^n)$ has 
a representation \eqref{deFinetti+-} with a 
signed measure $\gl$ on $\cP(S)$
such that
\begin{align}\label{texex}
  \norm{\gl}_{\cM(\cP(S))}\le \kknN.
\end{align}
The constant $\kknN$ given in \eqref{kknN} is, in general, the best
possible. 
\item \label{Texexm}
If furthermore $S$ is finite with $|S|=m$, 
then \eqref{texex} can be replaced by
\begin{align}\label{texexm}
\norm{\gl}_{\cMPS}\le 
\kknNm:=
\max_{\xx\in S^N}\normpispq{\gfxx\nn(\chi_{n,N})}{\ell_1(S)}
\end{align}
Moreover, this constant is the best possible for the given $S$.
If\/ $m\ge N$, then $\kknNm=\kknN$.
  \end{thmenumerate}
\end{theorem}

\begin{proof}
The proof of \refT{Tex} extends with minor changes as follows; we omit some
details. 
  
\pfitemref{Texexa}
 Fix a representation
  \begin{align}\label{fixedN}
    \chi_{n,N}
=\sumkM a_k \ff_{k}\nn,
  \end{align}
where $a_k\in\bbR$ and $\ff_k\ge0$ are unit vectors in
$\ell_1^N=\cM(\setNN)$
and $K:=\sum_k|a_k|=\kknN$ (see \refR{R=+}).
Thus,  $\ff_k\in\cP(\setNN)$.
Define again $\psi_{\xx}$ by \eqref{psixx}.
Then, similarly to \eqref{aslog},
\begin{align}\label{kik}
  \int_{\cP(S)}\nu^n\dd\psi_{\xx}(\nu)
&=\gfxx\snn(\chi_{n,N})
=\PiNn\bigpar{\gd_{x_1}\vees\gd_{x_N}}
\notag\\&
=\frac{1}{N!}\sum_{\gs\in\fS_N} \gd_{x_{\gs(1)}}\tensor\dotsm\tensor \gd_{x_{\gs(n)}}.
\end{align}
Again, $\psi_{\xx}$ is a bounded random measure in $\cMPS$, and we define
$\gl:=\E\psi_{\XX}\in\cMPS$.
Then $\norm{\gl}_{\cMPS}\le K=\kknN$, so \eqref{texex} holds, and similarly
to \eqref{disa}, using \eqref{kik} and exchangeability,
\begin{align}
  \int_{\cP(S)}\nu^n\dd\gl(\nu) &
=\E\Bigpar{
\frac{1}{N!}\sum_{\gs\in\fS_N} \gd_{X_{\gs(1)}}\tensors \gd_{X_{\gs(n)}}}
\notag\\&
=\E \bigpar{ \gd_{X_{1}}\tensors \gd_{X_{n}}}
=\muXX.
\end{align}

The case $\muXX=\chi_{n,N}$ shows that the constant $\kknN$ is best
possible, using \eqref{hedemora} as earlier.

\pfitemref{Texexm}
When $S$ is finite, every distribution in $\cE_{n,N}$ is a convex
combination 
of the
distributions $\gfxx\nn(\chi_{n,N})$, and thus \eqref{texexm} follows from
\eqref{puh}.

Conversely, each $\gfxx\nn(\chi_{n,N})\in\cE_{n,N}$, and thus \eqref{puh}
shows that \eqref{texexm} is best possible.
\end{proof}

\begin{remark}  \label{Rkkmono}
Since $(N+1)$-\ext{} implies $N$-\ext, it follows from \refT{Texex} that
\begin{gather}\label{kkmono}
  \kkn=\kk(n,n)\ge\kk(n,n+1)\ge\dots \ge1,
\\
  \kk(n,N;m)\ge\kk(n,n+1;m)\ge\dots \ge1.\label{kkmonom}
\end{gather}
One can also see \eqref{kkmono} directly from \eqref{kknN}, 
since $\chi_{n,N+1}$ is the average of
$\gfxx\nn(\chi_{n,N})$ over all sequences $\xx$ of $N$ distinct elements of
$[N+1]$.
\end{remark}

\citet{DiaconisF} showed (with  precise estimates)
that if $N$ is large, then a distribution
$\muXX\in\cE_{n,N}$ is close to a distribution as in \eqref{deFinetti},
in the sense of total variation.
This implies similar results in terms of the constants in \refT{Texex}.
In particular, for fixed $n$, the following theorem shows  
that $\kknN\to1$ as \Ntoo{}; more precisely, 
$\kknN=1+O(1/N)$ for fixed $n$, and this rate is exact. However, there is
a wide gap between the ``constants'' (depending on $n$) in the upper
and lower bounds given by the theorem. 

\begin{theorem}\label{Tkknn}
  \begin{thmenumerate}
  \item \label{Tkknn<}
If\/ $N>  n(n-1)/2$, then
\begin{align}\label{nua}
  \kknN\le 1+\frac{n(n-1)}{2N-n(n-1)}\bigpar{\kkn+1}.
\end{align}
\item \label{Tkknn>}
If\/ $N\ge n$, then
\begin{align}\label{nub}
  \kknN
\ge 
e^{ \frac{n-1}{2\ceil{N/n}}}
\ge 
1+\frac{n(n-1)}{2(N+n)}.
\end{align}
\item \label{Tkknnm<}
If\/ $N\ge n\ge m$, then
\begin{align}\label{nuam}
  \kk(n,N;m)
\le 1+\csp(n,\ell_1^m) \frac{2mn}{N}
\le 1+ \frac{2mn\kkn}{N}.
\end{align}
\item \label{Tkknnm>}
If\/ $N\ge n\ge m$, then
\begin{align}\label{nubm}
  \kk(n,N;m)
\ge 
e^{ \frac{m-1}{2\ceil{N/n}}}
\ge 
1+\frac{(m-1)n}{2(N+n)}.
\end{align}
  \end{thmenumerate}
\end{theorem}

\begin{proof}
\pfitemref{Tkknn<}
  Let $\nu_N$ be the uniform distribution on $\setNN$. Then $\nu_N^n$
is the  distribution of a random vector $X_1,\dots,X_n$ obtained by drawing 
randomly from  $\setNN$ with replacement. Conditioned on the event $\cD$ that
$X_1,\dots,X_N$ are distinct, this yields the distribution $\chi_{n,N}$.
Hence, if $q:=\P(\cD)$, then, \cf{} \cite{DiaconisF},
\begin{equation}
  \nu_N^n=q \chi_{n,N}+(1-q)\mu',  
\end{equation}
for some probability measure $\mu'\in\cP(\setNN^n)$.
Clearly, $\mu'$ is symmetric, \ie, exchangeable.
Consequently,
\begin{align}
  q\normpisp{\chi_{n,N}}&
\le
  q\normpisp{\nu_N^n}
+(1-q)\normpisp{\mu'}
\notag\\&
\le 1+(1-q)\csp(n,\ell_1^N)\normpi{\mu'}
=
 1+(1-q)\kkn
\end{align}
and thus
\begin{align}
\kknN= \normpisp{\chi_{n,N}}&
\le
 1+\bigpar{\kkn+1}\frac{1-q}{q}.
\end{align}
Furthermore,
\begin{align}\label{q}
  q=\P(\cD)=\prod_{i=1}^{n-1}\Bigpar{1-\frac{i}{N}}
\ge1- \sum_{i=1}^{n-1}\frac{i}{N}
=1-\frac{n(n-1)}{2N}.
\end{align}
Hence, \eqref{nua} follows

\pfitemref{Tkknn>}
We modify \refE{El1}.
Partition $\setNN$ into $n$ sets $S_1,\dots,S_n$
and let $N_i:=|S_i|$. 
Define a multilinear operator $L:(\ell_1^N)^n\to\bbR$
by, writing $x_i=(x_{ij})_{j=1}^N$,
\begin{align}
  L(x_1,\dots,x_n)
=\prod_{i=1}^n \sum_{j\in S_i} x_{ij}.
\end{align}
Regarding $L$ as a 
linear operator $L:(\ell_1^N)\nn\to\bbR$, we then have, 
if $X_1,\dots,X_n$ is a random vector with distribution $\chi_{n,N}$,
\begin{align}
  L(\chi_{n,N})
&=\E L(\gd_{X_1},\dots,\gd_{X_n})
= \E \prod_{i=1}^n \indic{X_i\in S_i}
=\frac{N_1}{N}\cdot\frac{N_2}{N-1}\dotsm\frac{N_n}{N-n+1}.
\end{align}
Furthermore, for any $x=(x_j)_1^N\in\ell_1^N$ with $\norm{x}\le1$, if
$s_i:=\sum_{j\in   S_i}|x_{j}|$, then by the arithmetic-geometric inequality,
\begin{align}
\bigabs{  L\bigpar{x\nn}}
\le s_1\dotsm s_n \le \Bigparfrac{\sum_i s_i}{n}^n \le n^{-n}.
\end{align}
Consequently, by \eqref{pris}, $\normpispx{L}\le n^{-n}$, and thus,
recalling \eqref{q},
\begin{align}\label{nuc}
  \kknN=\normpisp{\chi_{n,N}}
\ge n^n   L(\chi_{n,N})
=\prodin \frac{nN_i}{N-i+1}
=\frac{1}{q}\prodin \frac{nN_i}{N}.
\end{align}

Suppose first that $N$ is a multiple of $n$; $N=\ell n$ for an integer
$\ell$. Then we may choose $N_i=N/n=\ell$ for each $i$,
and thus \eqref{nuc} yields, 
\begin{align}\label{nue}
\log\kknN
\ge
  -\log q 
=-\sum_{i=1}^{n-1}\log\Bigpar{1-\frac{i}{N}}
\ge \sum_{i=1}^{n-1}\frac{i}{N}
=\frac{n(n-1)}{2N}.
\end{align}

For a general $N\ge n$ we let $\ell:=\ceil{N/n}$ and $N_1:=\ell n$.
Then $N\le N_1<N+n$, and \eqref{nue} yields, using \eqref{kkmono},
\begin{align}
  \log\kknN\ge \log\kk(n,N_1) \ge \frac{n(n-1)}{2N_1}
=\frac{n-1}{2\ceil{N/n}}
\ge \frac{n(n-1)}{2(N+n)}
\end{align}
and \eqref{nub} follows.

\pfitemref{Tkknnm<}
Suppose that $\muXX\in\cE_{n,N}$. Then, by \citet[Theorem (3)]{DiaconisF},
there exists a probability measure $\gl$ such that if
$\mu_0:=\int_{\cP(S)}\nu^n\dd\gl(\nu)$, then $\norm{\muXX-\mu_0}\le 2mn/N$.
Consequently, 
\begin{align}
  \normpispq{\muXX}{\ell_1(S)}&
\le\normpisp{\mu_0}+\normpisp{\muXX-\mu_0}
\le 1+\csp\xpar{n,\ell_1(S)} \normpi{\muXX-\mu_0}
\notag\\&
\le 1+\csp(n,\ell_1^m) \frac{2mn}{N}.
\end{align}
The result follows by \eqref{puh} and \refT{Tcsp}.

 \pfitemref{Tkknnm>}
We modify \refE{El1} again. We may assume $S=[m]$.
Let $n_1,\dots,n_m$ and $N_1,\dots,N_m$ be positive integers with 
$\sum_1^m n_k=n$ and $\sum_1^m N_k=n$.
Partition $\setNN$ and $\setnn$ into sets $S_k$ and $T_k$, respectively,
with $|S_k|=N_k$ and $|T_k|=n_k$.
Define a multilinear operator $L:(\ell_1^m)^n\to\bbR$
by, writing $x_i=(x_{ij})_{j=1}^m$,
\begin{align}
  L(x_1,\dots,x_n)
=\prod_{k=1}^m \prod_{i\in T_k} x_{ik}.
\end{align}
If $x=(x_j)_1^m\in\ell_1^m$ with $\norm{x}\le1$, then
by the arithmetic-geometric inequality,
\begin{align}
\bigabs{  L\bigpar{x\nn}}
=\prod_{k=1}^m|x_k|^{n_k}
=\prod_{k=1}^mn_k^{n_k}\prod_{k=1}^m\Bigpar{\frac{|x_k|}{n_k}}^{n_k}
\le \prod_{k=1}^mn_k^{n_k}\Bigpar{\frac{\norm{x}}{n}}^{n}.
\end{align}
Consequently, by \eqref{pris}, 
\begin{align}\label{dixi}
\normpispx{L}\le n^{-n}  \prod_{k=1}^mn_k^{n_k}.
\end{align}
Let $\xx=(x_1,\dots,x_N)\in[m]^N$ with $x_i=k$ when $i\in S_k$,
and let $(X_1,\dots,X_n)$ be a random vector 
obtained by drawing without replacement from $x_1,\dots,x_N$.
Then $(X_1,\dots,X_n)$ has
distribution
$\gfxx\nn\chi_{n,N}$, and thus, with the notation $(N)_n:=N(N-1)\dotsm(N-n+1)$,
\begin{align}\label{dixis}
  L\bigpar{\gfxx\nn\chi_{n,N}}
&=\E L(\gd_{X_1},\dots,\gd_{X_n})
= \E \prod_{k=1}^m \prod_{i\in T_k} \indic{X_i=k}
\notag\\&
=\frac{(N_1)_{n_1}\dotsm(N_m)_{n_m}}{(N)_n}
.\end{align}
Consequently, by \eqref{texexm}, \eqref{dixi} and \eqref{dixis},
\begin{align}\label{dixit}
\kk(n,N;m)&
\ge
\normpisp{\gfxx\nn\chi_{n,N}}
\ge \frac{L(\gfxx\nn\chi_{n,N})}{\normpispx{L}}
\ge
\frac{n^n\prod_{k=1}^m(N_k)_{n_k}}{(N)_n\prod_{k=1}^mn_k^{n_k}}
\notag\\&
=\frac{\prodkm\prod_{j=1}^{n_k-1}\lrpar{1-j/N_k}}
  {\prod_{j=1}^{n-1}\bigpar{1-\xfrac{j}{N}}} 
\prodkm\Bigparfrac{nN_k}{N n_k}^{n_k}
\end{align}

Suppose first again that $N=\ell n$ is a multiple of $n$. 
Then, given any $n_1,\dots,n_k$ with sum $n$, 
we may choose $N_k=\ell n_k$ for each $k$.
Then the final product in \eqref{dixit} is 1, and \eqref{dixit} yields, 
using \refL{LL} below with $t=1/\ell$,
\begin{align}\label{diximus}
\log\kk(n,N;m)
\ge
\frac{m-1}{2\ell}
=\frac{(m-1)n}{2N}.
\end{align}

For a general $N\ge n$ we let $\ell:=\ceil{N/n}$ and $N_1:=\ell n$.
Then $N\le N_1<N+n$, and \eqref{diximus} yields
\begin{align}
  \log\kk(n,N;m)\ge \log\kk(n,N_1;m) 
\ge 
\frac{m-1}{2\ceil{N/n}}
\ge \frac{(m-1)n}{2(N+n)}.
\end{align}
and \eqref{nubm} follows.
\end{proof}

\begin{lemma}\label{LL}
Let $n\ge m\ge1$ and let $n_1,\dots,n_m$ be positive integers with $\sum_1^m
n_k=n$. 
Then, for every $t\in\oi$,
\begin{align}
  \sumkm \sum_{i=0}^{n_k-1}\log\Bigpar{1-t\frac{i}{n_k}}
- \sum_{i=0}^{n-1}\log\Bigpar{1-t\frac{i}{n}}
\ge \frac{m-1}{2}t.
\end{align}
\end{lemma}

\begin{proof}
  Define two positive measures  on $\oio$ by
  \begin{align}
    \nu_1&:=\sumkm\sum_{i=0}^{n_k-1}\gd_{\xfrac{i}{n_k}},
&
    \nu_2&:=\sum_{i=0}^{n-1}\gd_{\xfrac{i}{n}}.
  \end{align}
Both $\nu_1$ and $\nu_2$ are integer-valued and have total mass $n$.
Furthermore, for any $x\in\oi$, the number of integers $i\ge0$ such that
$i/n_k<x$ equals $\ceil{n_kx}$. Hence,
\begin{align}
  \nu_1[0,x) &= \sumkm\ceil{n_kx}
\ge\sumkm n_kx =nx,
\\
\nu_2[0,x)&=\ceil{nx}.
\end{align}
Since $\nu_1$ is integer-valued, it follows that
\begin{align}
  \nu_1[0,x)\ge \ceil{nx}=\nu_2[0,x),
\qquad x\in\oi.  
\end{align}
This implies, by a standard argument using integration by parts,
that if $f(x)$ is any decreasing function on $\oio$, then
\begin{align}\label{credo}
  \intoi f(x)\dd\nu_1(x)
\ge
  \intoi f(x)\dd\nu_2(x).
\end{align}
Choose $f(x):=\log(1-tx)+tx$. Then \eqref{credo} implies
\begin{align}
&  \sumkm \sum_{i=0}^{n_k-1}\log\Bigpar{1-t\frac{i}{n_k}}
- \sum_{i=0}^{n-1}\log\Bigpar{1-t\frac{i}{n}}
\notag\\&\qquad
=\intoi\log(1-tx)\bigpar{\ddx\nu_1(x)-\ddx\nu_2(x)}
\notag\\&\qquad
\ge
-\intoi tx\bigpar{\ddx\nu_1(x)-\ddx\nu_2(x)}
=
 - \sumkm \sum_{i=0}^{n_k-1}{t\frac{i}{n_k}}
+ \sum_{i=0}^{n-1}{t\frac{i}{n}}
\notag\\&\qquad
=
 -t \sumkm {\frac{n_k-1}2}
+ t\frac{n-1}2
=t\frac{m-1}2
 .
\end{align}
\end{proof}

\begin{remark}\label{Rnub}
The proofs of the lower bounds in \eqref{nub} and \eqref{nubm} 
really yields lower bounds for $\normpis\,$ 
and not just the larger $\normpisp\,$ in
\eqref{nuc} and \eqref{dixit}.
Hence, the lower bounds cannot be expected to be close to the true values.
\end{remark}

\section{Further examples}\label{Sex}

\begin{example}\label{El2+}
  Let $E=\ell_2^2$, \ie, $\bbR^2$ with the usual Euclidean norm.
If $A\in E\snnx2$, so $A$ is a symmetric $2\times2$ matrix, then, by
\refEs{Ematrix} and \ref{EBanach},
\begin{align}\label{bc}
  \normpi{A}=\normpis{A}=|\gl_1|+|\gl_2|,
\end{align}
where $\gl_1,\gl_2$ are the eigenvalues of $A$.

In particular, taking $A:=\smatrixx{0&1\\1&0}$,
\begin{align}\label{bd}
  \bignormpi{\smatrixx{0&1\\1&0}}
=
  \bignormpis{\smatrixx{0&1\\1&0}}
=2.
\end{align}

Furthermore, $A=e_1\tensor e_2+e_2\tensor e_1=2e_1\vee e_2$, and thus
\eqref{normpip} (or \eqref{normpip2}) yields, together with \eqref{lpip},
\begin{align}\label{be}
 \bignormpip{\smatrixx{0&1\\1&0}}=2.  
\end{align}

A positive unit vector in $\ell_2^2$ is $(\cos t,\sin t)$ for some
$t\in\opii$. 
Hence, a representation of $A$ as in \eqref{normpisp} can be written
\begin{align}\label{ba}
  A = \int_0^{\pi/2}\xpar{\cos t,\sin t}\nnx2\dd\mu(t)
\end{align}
for a signed measure $\mu$ on $\opii$
with finite support.
Thus, $\normpisp{A}$ is the infimum of $\norm\mu$ over all such $\mu$
satisfying \eqref{ba}.

With $A=\smatrixx{0&1\\1&0}$ as above, \eqref{ba} says
\begin{align}
\int_0^{\pi/2}\cos^2 t\dd\mu(t)  
= \int_0^{\pi/2}\sin^2 t\dd\mu(t)  
=0,
&&&
\int_0^{\pi/2}\cos t\sin t\dd\mu(t)  =1,
\end{align}
and thus
\begin{align}
  \int_0^{\pi/2}\bigpar{1-2\sin2t}\dd\mu(t)  
=
\int_0^{\pi/2}\bigpar{\cos^2 t+\sin^2t-4\sin t\cos t}\dd\mu(t)  
=-4.
\label{bb}
\end{align}
Since $\bigabs{1-2\sin2t}\le1$ on $\opii$, \eqref{bb} implies
$\norm{\mu}\ge4$, which is attained by $\mu=2\gd_{\pi/4}-\gd_0-\gd_{\pi/2}$.
Hence,
\begin{align}\label{bh}
 \bignormpisp{\smatrixx{0&1\\1&0}}=4  
> \bignormpip{\smatrixx{0&1\\1&0}}=2. 
\end{align}

In particular, \eqref{bh} shows that the result by \citet{Banach}, see again
\refE{EBanach}, does not extend to the positive tensor norms.
\end{example}

\begin{example}\label{El2g}
  Consider as in the previous example $E=\ell_2^2$. The different norms in
  $E\snnx2$ can be described geometrically as follows.

We give a matrix $A\in E\snnx2$ the coordinates $(u,v,w)$ defined by
\begin{align}\label{bff}
  A = \frac12
  \begin{pmatrix}
u+w & v \\ v & u-w
  \end{pmatrix}.
\end{align}
%
In these coordinates, we have $(\cos t,\sin t)\nnx2=(1,\sin 2t, \cos 2t)$.

The unit ball of $E\snnpix2=E\snnpisx2$ (see \refE{EBanach} again)
is by 
\refR{R=s} thus the convex hull of 
\begin{align}\label{bfpi}
\bigset{\pm(1,\sin 2t, \cos 2t):t\in[0,2\pi]}
=
\bigset{\pm(1,\sin s, \cos s):s\in[0,2\pi]} .
\end{align}
This is the convex hull of the union of two
symmetric circles, and thus the unit ball is the cylinder
$\set{|u|\le 1,\, v^2+w^2\le1}$.
In other words, $\normpis{(u,v,w)}=\max\bigset{|u|,\sqrt{v^2+w^2}}$,
which also easily is seen from \eqref{bc}.

For $E\snnpispx2$
we are by \eqref{normpisp}
only allowed to use positive vectors $(\cos t,\sin t)$, \ie, $t\in\opii$.
Consequently, the unit ball of $E\snnpispx2$ is the convex hull of the union
of two symmetric half-circles:
\begin{align}\label{bfpisp}
  \bigset{\pm(1,\sin s, \cos s):s\in[0,\pi]} .
\end{align}

Finally, in our coordinates, 
$(\cos s,\sin s)\vee(\cos t,\sin t)=
\bigpar{\cos(s-t),\sin(s+t),\cos(s+t)}$.
When $s,t\in\opii$, we have $s+t\in[0,\pi]$ and $|s-t|\le
\min\set{s+t,\pi-s-t}$. It follows from \eqref{normpip} that
the unit ball of $E\snnpipx2$ is the convex hull of the union of 
two half-circles (the same as in \eqref{bfpisp}) 
and four elliptic arcs given by
\begin{align}\label{bfpip}
  \bigset{\pm(1,\sin s, \cos s):s\in[0,\pi]} \cup
  \bigset{\pm(|\cos s|,\sin s,\cos s):s\in[0,\pi]}.
\end{align}

Note that the three sets in \eqref{bfpi}, \eqref{bfpisp} and \eqref{bfpip} are
the sets of extreme points of the unit balls.

To help visualizing these three unit balls, we consider their orthogonal
projections onto the plane $Q:=\set{w=0}$, which are the same as their
intersections with $Q$ since they all are symmetric with respect to
reflection in this plane. It follows easily from
\eqref{bfpi}, \eqref{bfpisp} and \eqref{bfpip} that these projections all are
polygons, with corners (extreme points)
\begin{align}
B(E\snnpix2)&:
\set{(\pm1,\pm1,0)},
\\
B(E\snnpispx2)&:
  \set{\pm(1,1,0),\,\pm(1,0,0)},
\\
B(E\snnpipx2)&:
  \set{\pm(1,1,0),\,\pm(1,0,0),\,\pm(0,1,0)}.
\end{align}
Equivalently, recalling \eqref{bff} and taking $u=2a$, $v=2b$, for any
$a,b\in\bbR$,
\begin{align}
\bignormpi{\smatrixx{a&b\\b&a}}&
=\bignormpis{\smatrixx{a&b\\b&a}}
=\normpi{(u,v,0)}
=\max\set{|u|,|v|}
=2\max\set{|a|,|b|},
\label{abpi}\\
\bignormpisp{\smatrixx{a&b\\b&a}}&
=  \normpisp{(u,v,0)}
= \max\set{|u|,|u-2v|}
=2 \max\set{|a|,|a-2b|},
\label{abpisp}\\
\bignormpip{\smatrixx{a&b\\b&a}}&
=  \normpip{(u,v,0)}
= \max\set{|u|,|v|,|u-v|}
= 2\max\set{|a|,|b|,|a-b|}.
\label{abpip}
\end{align}
In particular, we find again \eqref{bd}, \eqref{be} and \eqref{bh}.

Conversely, \eqref{abpi}--\eqref{abpip} can be found by the analytic method
in \refE{El2+}.

We find also, as another specific example,
\begin{align}
&  \bignormpi{\smatrixx{\phantom-1&-1\\-1&\phantom-1}}
= \bignormpis{\smatrixx{\phantom-1&-1\\-1&\phantom-1}}
=2,
\label{byxpi}\\
&  \bignormpisp{\smatrixx{\phantom-1&-1\\-1&\phantom-1}}
=6,
\label{byxpisp}\\
 & \bignormpip{\smatrixx{\phantom-1&-1\\-1&\phantom-1}}
=4.
\label{byxpip}
\end{align}
We claim that
\begin{align}
\csp(2,\ell_2^2)= \cssp(2,\ell_2^2)=3.
\end{align}
In fact, the two polarization constants are equal by \eqref{abcH}.
They are at least 3 by \eqref{byxpi} and \eqref{byxpisp}.
Finally, to show that they are at most 3, it suffices 
by \eqref{bfpi} and \eqref{bfpisp}
to consider 
$\xx=(1,\sin s\cos s)$ with $s\in(\pi,2\pi)$.
Then, with $s':=s-\pi$,
\begin{align}
  \xx=-(1,\sin s',\cos s')+(1,0,1)+(1,0,-1),
\end{align}
which by \eqref{bfpisp} shows that $\normpisp{\xx}\le3$;
it then follows from \eqref{cspn} that $\csp(2,\ell_2^2)\le3$.

By a similar argument, using \eqref{bh} for the lower bound
and \eqref{bfpip} and \eqref{bfpisp} for the upper,
we obtain (omitting the details), recalling \eqref{cpspn},
\begin{align}
\cpsp(2,\ell_2^2)=2.
\end{align}

We can also see that, as shown in \eqref{cpip}, the norm of the identity
$E\snnpix2 \to E\snnpipx2$ is $\cp(\ell_2^2)^2=2$, \cf{} \refE{Ecp}. 
\end{example}

\ack{I thank Takis Konstantopoulos for interesting discussions.}

\appendix

\section{Linear polarization constants}\label{Alinear}

We review in this appendix for comparison some results on another
``polarization constant'' that also has been studied. As far as we know,
there are no direct relations with the constants above. We nevertheless find
it interesting to compare the results and see similarities and differences.

Let $f_1,\dots,f_n\in E^*$, the dual of $E$.
Then
\begin{equation}\label{Ltensor}
  L(x_1,\dots,x_n):=\prodin f_i(x_i)
\end{equation}
defines an $n$-linear form on $E$, denoted by $f_1\tensor\dotsm\tensor f_n$.
In this case, the corresponding polynomial
$\hL$ is simply 
\begin{equation}\label{hLtensor}
  \hL(x)=\prodin f_i(x),
\end{equation}
\ie, $\hL=\prodin f_i$.
We have, as immediate consequences of \eqref{Ltensor} and \eqref{hLtensor},
\begin{align}
  \norm{L}&=\prodin\norm{f_i},
\\
\normq{L}&=\norm{\hL}=\Bignorm{\prodin f_i}.
\end{align}

Following \citet{BST},
we make the following definition in analogy with \refD{Dc}, 

\begin{definition}\label{Dcl}
  The \emph{linear polarization constant} $\cl(n,E)$ is defined by
  \begin{equation}\label{cl}
	\cl(n,E):=\sup_{L=f_1\tensors f_n}\frac{\norm{L}}{\norm{\hL}}
=\sup_{f_1,\dots,f_n\in E^*}\frac{\norm{f_1}\dotsm\norm{f_n}}
{\norm{\prodin f_i}}.
  \end{equation}
 \end{definition}
Equivalently,
  \begin{equation}\label{cl2}
	\cl(n,E)\qw
=\inf\Bigset{{\Bignorm{\prodin f_i}}:\norm{f_i}_{E^*}=1, i=1,\dots,n}.
  \end{equation}
In other words, $\cl(n,E)\qw$ is the largest number such that for any
$f_1,\dots,f_n\in E^*$,
\begin{equation}\label{clA7}
  \sup_{\norm{x}\le1}\prodin \bigabs{f_i(x)}
\ge \cl(n,E)\qw \prodin\norm{f_i}_{E^*}.
\end{equation}
As said above,
there are no direct relations with the constants above.  
Note that both \eqref{c} and \eqref{cl} are suprema of the same ratio of norms
$\norm L/\norm{\hL}=\norm{L}/\normq{L}$
over some sets of multilinear forms $L$
(the set is a linear space in \eqref{c} but not, in general, in \eqref{cl});
however, neither set includes the
other (for $n\ge2$). (The functions $L=f_1\tensors f_n$ in
\eqref{cl} typically are not  symmetric, and a general symmetric
$n$-linear form $L$ in \eqref{c} typically is not an elementary tensor
$f_1\tensors f_n$.)

Clearly, $\cl(n,E)\ge1$. It is proved in \cite{RyanT} that $\cl(n,E)<\infty$
for any $n$ and $E$. Moreover,
by \cite{BST} (complex case) and \cite{Korean} (real case, as a consequence
of \cite{Ball:plank}),
\begin{equation}\label{clnn}
  \cl(n,E)\le n^n
\end{equation}
for any normed space $E$. 

\begin{example}\label{Ecll1}
For any $m\ge n$,
$\cl(n;\ell_1^m)=\cl(n;\ell_1)=n^n$, by \eqref{clnn} and
the same example \eqref{el1} as in \refE{El1}, \ie, taking
$f_i$ as the $i$-th coordinate function and using \eqref{eli}. 
(See \cite{BST}.)
Hence, equality can hold  in \eqref{clnn}.
\end{example}

\begin{remark}\label{Rclfunc}
  It is easy to see that \refL{Lfunc} holds for $\cl$ too; hence, as noted
  in \cite{BST},
all parts of  \refT{Tfunc} holds for $\cl$ too.
\end{remark}

\begin{example}\label{EclH}
  For a complex Hilbert space $H$, as proved by \cite{Arias} and
  \cite{Ball:complex}, 
  \begin{equation}
	\cl(n,H)\le n^{n/2},	
  \end{equation}
with equality if $\dim(H)\ge n$. (The lower bound is obtained by taking
$f_1,\dots,f_n$ orthogonal.)
For real Hilbert spaces, 
the same is conjectured 
but so far proved only for $\dim(H)\le5$ \cite{PappasR}; 
for upper bounds see
\cite{Korean, 
Frenkel,
MunozEtal2010};
again the lower bound $\cl(n,H)\ge n^{n/2}$ holds if $\dim(H)\ge n$.

Furthermore, \cite{Korean} proved, using a dual version of Dvoretzky's
theorem and the result by \cite{BST} mentioned in \refR{Rclfunc},  
that $\cl(n,E)\ge \cl(n,\ell_2^n)$ for any infinite-dimensional
Banach space. (And thus for every infinite-dimensional normed space, since
$\cl(n,E)=\cl(n,\bar E)$ if $\bar E$ is the completion of $E$.)
Consequently, for any infinite-dimensional normed space $E$,
\begin{equation}\label{super}
  \cl(n,E)\ge n^{n/2}.
\end{equation}
\end{example}

For further examples, see \cite{Korean}.

\begin{remark}\label{Rsuper}
  It is easily seen from the definition
that $\cl(n,E)$ is supermultiplicative:
\begin{equation}
  \cl(m+n,E)\ge\cl(m,E)\cl(n,E),
\end{equation}
for any normed
  space $E$, see \cite{Korean}.
As a consequence, the limit
\begin{equation}
  \cl(E):=\lim_\ntoo \cl(n,E)^{1/n}
= \sup_n \cl(n,E)^{1/n}\in[1,\infty]
\end{equation}
exists, \cf{} \eqref{rc}.
For a calculation of $\cl(\bbR^d)$ and $\cl(\bbC^d)$, see \cite{PappasR}.
Further results are given in \eg{} \cite{CarandoEtal}.
\end{remark}

\begin{remark}
  It is seen from \eqref{super} that $\cs(n,E)$ grows faster than
  exponentially when $E$ is infinite-dimensional, so $\cl(E)$ may be
  infinite. 
In fact, see  \cite{Korean},
$\cl(E)<\infty\iff \dim(E)<\infty$.
\end{remark}

\begin{remark}
  \citet{BST} proved also more general results on 
$\sup \xfrac{\norm{q_1}\dotsm\norm{q_n}}{\norm{\prodin q_i}}$ where $q_i$
are polynomials of given degrees $m_1,\dots,m_n$, 
obtaining an extension of \eqref{clA7}
with a different constant depending on $m_1,\dots,m_n$ replacing $\cl(n,E)$.
See further \eg{} \cite{Pinasco}.
\end{remark}

\newcommand\AAP{\emph{Adv. Appl. Probab.} }
\newcommand\JAP{\emph{J. Appl. Probab.} }
\newcommand\JAMS{\emph{J. \AMS} }
\newcommand\MAMS{\emph{Memoirs \AMS} }
\newcommand\PAMS{\emph{Proc. \AMS} }
\newcommand\TAMS{\emph{Trans. \AMS} }
\newcommand\AnnMS{\emph{Ann. Math. Statist.} }
\newcommand\AnnPr{\emph{Ann. Probab.} }
\newcommand\CPC{\emph{Combin. Probab. Comput.} }
\newcommand\JMAA{\emph{J. Math. Anal. Appl.} }
\newcommand\RSA{\emph{Random Struct. Alg.} }
\newcommand\ZW{\emph{Z. Wahrsch. Verw. Gebiete} }
\newcommand\DMTCS{\jour{Discr. Math. Theor. Comput. Sci.} }

\newcommand\AMS{Amer. Math. Soc.}
\newcommand\Springer{Springer-Verlag}
\newcommand\Wiley{Wiley}

\newcommand\vol{\textbf}
\newcommand\jour{\emph}
\newcommand\book{\emph}
\newcommand\inbook{\emph}
\def\no#1#2,{\unskip#2, no. #1,} 
\newcommand\toappear{\unskip, to appear}

\newcommand\arxiv[1]{\texttt{arXiv:#1}}
\newcommand\arXiv{\arxiv}

\def\nobibitem#1\par{}


\begin{thebibliography}{99}
\renewcommand\and{\& }

\bibitem[Aldous(1985)]{Aldous}
David J. Aldous:
\emph{Exchangeability and related topics.} 
{\'E}cole d'{\'e}t{\'e} de probabilit{\'e}s de Saint-Flour, XIII--1983, 1--198,
Lecture Notes in Math., 1117, Springer, Berlin, 1985. 

\bibitem{Arias}
J. Arias-de-Reyna:
Gaussian variables, polynomials and permanents.
\emph{Linear Algebra Appl.} \vol{285} (1998), no. 1-3, 107--114. 

\bibitem{Ball:plank}
Keith Ball:
The plank problem for symmetric bodies.
\emph{Invent. Math.} \vol{104} (1991), no. 3, 535--543.

\bibitem{Ball:complex}
Keith M.  Ball:
The complex plank problem. 
\emph{Bull. London Math. Soc.} \vol{33} (2001), no. 4, 433--442. 

\bibitem[Banach(1938)]{Banach}
S. Banach:
{\"U}ber homogene Polynome in $(L^2)$.
\emph{Studia Math.}
\textbf7 (1938),
36--44.

\bibitem[Ben{\'\i}tez,  Sarantopoulos and  Tonge(1998)]{BST}
Carlos Ben{\'\i}tez, Yannis Sarantopoulos \and Andrew Tonge:
Lower bounds for norms of products of polynomials.
\emph{Math. Proc. Cambridge Philos. Soc.} \vol{124} (1998), no. 3, 395--408. 

\bibitem[Bu and Buskes(2012)]{BuBuskes}
Qingying Bu \and Gerard Buskes:
Polynomials on Banach lattices and positive tensor products.
\emph{J. Math. Anal. Appl.} \vol{388} (2012), no. 2, 845--862. 

\bibitem{CarandoEtal}
Daniel Carando, Dami{\'a}n Pinasco \and Jorge Tom{\'a}s Rodr{\'i}guez:
On the linear polarization constants of finite dimensional spaces. 
\emph{Math. Nachr.} \vol{290} (2017), no. 16, 2547--2559.

\bibitem{DM} 
Claude Dellacherie \and Paul-Andr\'e Meyer:
\emph{Probabilities and Potential B}. 
(Translated from French.)
North-Holland, Amsterdam,
1982.


\bibitem[Diaconis(1977)]{Diaconis}
Persi Diaconis:
Finite forms of de Finetti's theorem on exchangeability.
\emph{Synthese} \vol{36} (1977), no. 2, 271--281. 

\bibitem[Diaconis and Freedman(1980)]{DiaconisF}
Persi Diaconis \and David Freedman: 
Finite exchangeable sequences.
\emph{Ann. Probab.} \vol8 (1980), no. 4, 745--764. 

\bibitem[Dineen(1999)]{Dineen}
Se{\'a}n Dineen:
\emph{Complex Analysis on Infinite Dimensional Spaces}.
Springer-Verlag 
London, 1999. 

\bibitem[Floret(1997)]{Floret}
Klaus Floret:
Natural norms on symmetric tensor products of normed spaces.
\emph{Note Mat.} \vol{17} (1997),  153--188.

\bibitem[Fremlin(1974)]{Fremlin}
D. H. Fremlin:
Tensor products of Banach lattices.
\emph{Math. Ann.} \vol{211} (1974), 87--106. 

\bibitem{Frenkel}
P{\'e}ter E. Frenkel:
Pfaffians, Hafnians and products of real linear functionals.
\emph{Math. Res. Lett.} \vol{15} (2008), no. 2, 351--358. 

\bibitem[Friedland and Lim(2018)]{Friedland}
Shmuel Friedland \and Lek-Heng Lim:
Nuclear norm of higher-order tensors.
\emph{Math. Comp.} \vol{87} (2018), no. 311, 1255--1281. 

\bibitem{GK}
I. C.  Gohberg \and  M. G. Kre{\u\i}n:
\emph{Introduction to the Theory of Linear Nonselfadjoint Operators}. 
(Translated from Russian.) 
\AMS{}, Providence, R.I., 1969.


\bibitem[Harris(1981)]{Harris}
Lawrence A. Harris:
Commentary on Problem 73,
\emph{The Scottish Book: Mathematics from the Scottish Caf\'e},
ed. R. Daniel Maudlin,
Birkh\"auser, Boston, 1981, 143--146.


\bibitem[Jaynes(1986)]{Jaynes}
Edwin T. Jaynes:
Some applications and extensions of the de Finetti representation theorem. 
\emph{Bayesian Inference and Decision Techniques}, 31--42,
North-Holland, Amsterdam, 1986. 

\bibitem[Janson, Konstantopoulos and Yuan(2016)]{SJ308}
Svante Janson, Takis Konstantopoulos \and  Linglong Yuan: 
On a representation theorem for finitely exchangeable random vectors.
\emph{J. Math. Anal. Appl.} \vol{442} (2016), 703--714.

\bibitem[Kallenberg(2005)]{Kallenberg-symmetries}
Olav Kallenberg:
\newblock \emph{Probabilistic Symmetries and Invariance Principles}.
\newblock Springer, New York, 2005.

\bibitem[Kallenberg(2017)]{Kallenberg:RM}
Olav Kallenberg:
\book{Random Measures, Theory and Applications}.
Springer, Cham, Switzerland, 2017. 

\bibitem[Kerns and  Sz\'ekely(2006)]{KS06}
G.\ Jay Kerns \and  G\'abor J.\  Sz\'ekely:
\newblock De Finetti's theorem for abstract finite exchangeable sequences, 
\emph{J. Theoret. Probab.} \vol{19} \no3 (2006), 589--608.

\bibitem[Konstantopoulos and Yuan(2015)]{Takis2}
Takis Konstantopoulos \and Linglong Yuan: 
On the extendibility of finitely exchangeable probability measures.
\arxiv{1501.06188} 

\bibitem{Lax}
Peter D. Lax:
\emph{Functional Analysis}. 
\Wiley, 
New York, 2002. 

\bibitem{Meyer-Nieberg}
Peter Meyer-Nieberg:
\emph{Banach Lattices}.
Springer-Verlag, Berlin, 1991. 

\nobibitem{Munoz}
Gustavo A. Mu{\~n}oz, 
Yannis Sarantopoulos 
and Andrew Tonge, 
Complexifications of real Banach spaces, polynomials and multilinear maps. 
\emph{Studia Math.} \vol{134} (1999), no. 1, 1--33.

\bibitem{MunozEtal2010}
G. A. Mu{\~n}oz-Fern{\'a}ndez, Y. Sarantopoulos \and J. B. Seoane-Sep{\'u}lveda:
The real plank problem and some applications.
\emph{Proc. Amer. Math. Soc.} \vol{138} (2010), no. 7, 2521--2535. 


\bibitem{NIST}
\emph{NIST Handbook of Mathematical Functions}. 
Edited by Frank W. J. Olver, Daniel W. Lozier, Ronald F. Boisvert \and
Charles W. Clark. 
Cambridge Univ. Press, 2010. \\
Also available as 
\emph{NIST Digital Library of Mathematical Functions},
\url{http://dlmf.nist.gov/}

\bibitem{PappasKK}
Alexandros Pappas,  Andreas Kavadjiklis \and Michael Karamolengos:
Polarization constants of polynomials on Banach spaces.
\emph{Nonlinear Funct. Anal. Appl.} \vol{14} (2009), no. 4, 551--562. 

\bibitem{PappasR}
Alexandros Pappas \and  Szil{\'a}rd Gy.  R{\'e}v{\'e}sz:
Linear polarization constants of Hilbert spaces. 
\emph{J. Math. Anal. Appl.} \vol{300} (2004), no. 1, 129--146.

\bibitem{Pinasco}
Dami{\'a}n Pinasco:
Lower bounds for norms of products of polynomials via Bombieri inequality.
\emph{Trans. Amer. Math. Soc.} \vol{364} (2012), no. 8, 3993--4010. 

\bibitem[Qi, Comon and Lim(2016a)]{QiEtal1}
Yang Qi,  Pierre Comon \and Lek-Heng Lim:
Uniqueness of nonnegative tensor approximations.
\emph{IEEE Trans. Inform. Theory} \vol{62} (2016), no. 4, 2170--2183. 

\bibitem[Qi, Comon and Lim(2016b)]{QiEtal2}
Yang Qi,  Pierre Comon \and Lek-Heng Lim:
Semialgebraic geometry of nonnegative tensor rank.
\emph{SIAM J. Matrix Anal. Appl.} \vol{37} (2016), no. 4, 1556--1580. 


\bibitem{Korean}
Szil{\'a}rd Gy.  R{\'e}v{\'e}sz \and Yannis Sarantopoulos:
Plank problems, polarization and Chebyshev constants.
\emph{J. Korean Math. Soc.} \vol{41} (2004), no. 1, 157--174. 

\bibitem[Rivlin(1974)]{Rivlin}
Theodore J. Rivlin:
\book{The Chebyshev Polynomials}.
Wiley, New York, 1974.

\bibitem{Rudin:FA}
Walter Rudin:
\emph{Functional Analysis}.
2nd ed., McGraw-Hill, New York, 1991.

\bibitem[Ryan(2002)]{Ryan}
Raymond A. Ryan:
\book{Introduction to Tensor Products on Banach Spaces}.
Springer-Verlag, London, 2002.

\bibitem{RyanT}
Raymond A. Ryan \and  Barry Turett:
 Geometry of spaces of polynomials. 
\emph{J. Math. Anal. Appl.} \vol{221} (1998), no. 2, 698--711.

\bibitem{Schaefer}
Helmut H.  Schaefer:
\emph{Banach Lattices and Positive Operators}. 
Springer-Verlag, New York-Heidelberg, 1974. 

\bibitem{Treves}
Fran{\c c}ois Tr{\`e}ves:
\emph{Topological Vector Spaces, Distributions and Kernels}. 
Academic Press, New York--London 1967. 






\end{thebibliography}
\end{document}